\documentclass[12pt,a4paper]{article}
\usepackage{amssymb,amsmath,amsthm,verbatim} 
\usepackage{times,cite,graphicx}
\usepackage{url}
\newtheorem{thm}{\sc Theorem.}[section]
\newtheorem{lem}[thm]{\sc Lemma.}
\newtheorem{cor}[thm]{\sc Corollary.}
\newtheorem{rem}[thm]{\sc Remark.}
\newenvironment{AMS}%
{{\upshape\bfseries AMS subject classifications. }\ignorespaces}{}
\newenvironment{keywords}{{\upshape\bfseries Key words. }\ignorespaces}{}
\newcommand{\R}{{\mathbb R}}

\newcommand{\dx}{\; {\rm d}x}
\newcommand{\ds}{\; {\rm d}s}
\renewcommand{\vec}{}
\newcommand{\unitn}{\vec{\rm n}}

\newcommand{\dd}[1]{\frac{\rm d}{{\rm d}#1}}
\newcommand{\ddt}{\dd{t}}
\def\epsilon{\varepsilon} 
\newcommand{\cPsi}{c_\Psi}
\newcommand{\wD}{w_{\partial\Omega}}
\newcommand{\ShD}{S^h_D}
\newcommand{\vol}{\operatorname{vol}}
\newcommand{\hatalpha}{\widehat\alpha} 

\def\vL{L\kern-0.08cm\char39}

\begin{document}
\title{On the stable discretization
of strongly anisotropic phase field models with applications to crystal growth}

\author{John W. Barrett\footnotemark[2] \and 
        Harald Garcke\footnotemark[3]\ \and 
        Robert N\"urnberg\footnotemark[2]}

\renewcommand{\thefootnote}{\fnsymbol{footnote}}
\footnotetext[2]{Department of Mathematics, 
Imperial College London, London, SW7 2AZ, UK}
\footnotetext[3]{Fakult{\"a}t f{\"u}r Mathematik, Universit{\"a}t Regensburg, 
93040 Regensburg, Germany}

\date{}

\maketitle

\begin{abstract}
We introduce unconditionally stable finite element approximations for
anisotropic Allen--Cahn and Cahn--Hilliard equations. These equations
frequently feature in phase field models that appear in materials science. 
On introducing the novel fully practical finite element approximations we prove
their stability and demonstrate their applicability with some numerical
results.

We dedicate this article to the memory of our colleague and friend 
Christof Eck (1968--2011) in recognition of his fundamental contributions to
phase field models.
\end{abstract} 


\renewcommand{\thefootnote}{\arabic{footnote}}

\section{Introduction} \label{sec:0}
The isotropic Cahn--Hilliard equation
\begin{equation} \label{eq:CH}
\theta\,\tfrac{\partial u}{\partial t}=
 \nabla \,.\,(\,b(u)\, \nabla\, w),\qquad 
 w = -\epsilon\,\Delta\, u + \epsilon^{-1}\,\Psi'(u)
\end{equation}
was originally introduced to model spinodal decomposition and coarsening 
phenomena in binary alloys, see \cite{Cahn61,CahnH58}.
Here $u$ is defined to be the difference of the 
local concentrations of the two components of an alloy
and hence $u$ is restricted to lie in the interval $[-1,1]$. More recently,
the Cahn--Hilliard equation has been used e.g.\ as a phase field
approximation  for
sharp interface evolutions and to study phase transitions and interface 
dynamics in multiphase fluids, see e.g.\
\cite{CaginalpC98,voids,KimKL04,AbelsGG12} 
and the references therein. We note that 
with $\theta = 1$ and $b(u)=1$ in (\ref{eq:CH}) 
in the limit $\epsilon\to0$,
we recover the well known sharp interface motion by Mullins--Sekerka, whereas
$\theta = \epsilon$ and $b(u)=1-u^2$ leads to 
surface diffusion; see below for details.

The theory of Cahn and Hilliard is based on the following Ginzburg--Landau
free energy 
$${\cal E}(u):= \int_\Omega \tfrac\epsilon2\, |\nabla\, u|^2 +
\epsilon^{-1}\,\Psi(u) \dx\,,
$$
where $\epsilon>0$ is a parameter and a measure for the interfacial thickness
and $\Omega \subset \R^d$, $d=2,3$, is a given domain.
The first term in the free energy penalizes large gradients
and the second term is the homogeneous free energy.
In this paper, we consider the so-called zero temperature ``deep quench''
limit, where a possible choice is
\begin{equation} 
\Psi(u):= \begin{cases}
\textstyle \frac 12 \left(1-u^2\right)  & |u|\leq 1\,,\\
\infty & |u|> 1\,,
\end{cases} 
\qquad\text{with}\quad
\cPsi := \int_{-1}^1 \sqrt{2\,\Psi(s)}\;{\rm d}s = \tfrac\pi2\,,
\label{eq:obst}
\end{equation}
see \cite{BloweyE92,BarrettBG99}. Clearly the obstacle potential 
$\Psi$ is not differentiable at
$\pm1$. Hence, whenever we write $\Psi'(u)$ in this paper we mean that the
expression holds only for $|u|<1$, and that in general a variational inequality
needs to be employed.

We note that (\ref{eq:CH}) can be derived from mass balance
considerations as a gradient flow for the free energy $\mathcal{E}(u)$,
with the chemical potential $w := \frac{\delta \mathcal{E}}{\delta u}$ 
being the variational derivative of the energy $\mathcal{E}$ with respect to 
$u$.
We remark that evolutions of (\ref{eq:CH}) lead to structures consisting of
bulk regions in which $u$ takes the values $\pm1$, and separating
these regions there will be interfacial transition layers across which $u$
changes rapidly from one bulk value to the other. 
With the help of formal asymptotics it can be shown that the width of
these layers is approximately $\epsilon\,\pi$; see e.g.\ 
\cite{Caginalp86,CahnEN-C96,GarckeNS98}. 

In this paper we want to consider an anisotropic variant of ${\cal E}(u)$, and
hence of (\ref{eq:CH}). To this end, we introduce the anisotropic density
function $\gamma : \R^d \to \R_{\geq 0}$ with 
$\gamma\in C^2(\R^d \setminus \{ 0 \}) \cap C(\R^d)$ which is assumed to be 
absolutely homogeneous of degree one, i.e.\
\begin{equation}
\gamma(\lambda\,\vec{p}) = |\lambda|\gamma(\vec{p}) \quad \forall \
\vec{p}\in \R^d,\ \forall\ \lambda \in\R \quad \Rightarrow
\quad \gamma'(\vec{p})\,.\,\vec{p} = \gamma(\vec{p})
\quad \forall\ \vec{p}\in \R^d\setminus\{\vec0\}, \label{eq:homo}
\end{equation}
where $\gamma'$ is the gradient of $\gamma$.
Then the anisotropy function defined as
\begin{equation}
A(\vec p) = \tfrac12\,|\gamma(\vec p)|^2 \qquad \forall\ \vec p \in \R^d\,,
\label{eq:Ap}
\end{equation}
is absolutely homogeneous of degree two and 
gives rise to the following anisotropic Ginzburg--Landau
free energy 
\begin{equation} \label{eq:Eg}
{\cal E}_\gamma(u) := \int_\Omega \tfrac\epsilon2\,|\gamma(\nabla\, u)|^2 +
\epsilon^{-1}\,\Psi(u) \dx
\equiv \int_\Omega \epsilon\, A(\nabla\, u) + 
\epsilon^{-1}\,\Psi(u) \dx
\,;
\end{equation}
see e.g.\ \cite{Elliott97,ElliottS96,DeckelnickDE05}. Note that
$\mathcal{E}_\gamma$ reduces to $\mathcal{E}$ in the isotropic case, i.e.\
when $\gamma$ satisfies
\begin{equation} \label{eq:iso}
\gamma(\vec p) = |\vec p| \qquad \forall\ \vec p \in \R^d\,.
\end{equation}
In this paper, we will only consider smooth and convex anisotropies, i.e.\
they satisfy
\begin{equation}
\gamma'(\vec{p})\,.\,\vec{q} \leq \gamma(\vec{q})
\qquad \forall\ \vec p \in \R^d\setminus\{\vec 0\}\,, \vec q \in \R^d\,,
\label{eq:gest}
\end{equation}
which, on recalling (\ref{eq:homo}), is equivalent to
\begin{equation} \label{eq:convex}
\gamma(\vec{p}) + \gamma'(\vec{p})\,.\,(\vec{q}-\vec{p})\leq\gamma(\vec{q})
\qquad \forall\ \vec p \in \R^d\setminus\{\vec 0\}\,, \vec q \in \R^d\,.
\end{equation}

Together with initial and natural boundary conditions, the anisotropic
Ginzburg--Landau energy (\ref{eq:Eg}) yields the following 
anisotropic Cahn--Hilliard equation:
\begin{subequations}
\begin{alignat}{2}
\theta\,
\frac{\partial u}{\partial t} & = 
\nabla \,.\, (b(u)\,\nabla\, w )
&& \mbox{in} \;\;\Omega_T:=\Omega\times(0,T)\,, \label{eq:CHa} \\ 
\tfrac12\,\cPsi\,\alpha^{-1}\,w & 
= -\epsilon\,\nabla \,.\, A'(\nabla\, u) +\epsilon^{-1}\,\Psi'(u) 
\qquad && \mbox{in} \;\;\Omega_T\,, \label{eq:CHb} \\ 
\frac{\partial u}{\partial \vec\nu} & = 0\,, \qquad
b(u)\,\frac{\partial w}{\partial \vec\nu} = 0 
&& \mbox{on}\;\;\partial \Omega\times(0,T)\,, \label{eq:CHc} \\
\qquad u(\cdot,0) & = u_0 && \mbox{in} \;\;\Omega\,, \label{eq:CHd}
\end{alignat}
\end{subequations}
where $\theta, \alpha \in \R_{>0}$ with
$\alpha$ being a factor relating to surface tension in the sharp interface 
limit,
and where $\vec\nu$ is the outer normal to $\partial \Omega$. Moreover,
$u_0 : \Omega \to \R$ is some initial data satisfying $|u_0| \leq 1$.

An alternative to the no-flux boundary conditions (\ref{eq:CHc}) are the
conditions
\begin{equation} \label{eq:CHDbc}
\frac{\partial u}{\partial \vec\nu} = 0\,, \qquad
w = g \qquad \mbox{on}\;\;\partial \Omega\times(0,T)\,, 
\end{equation}
which are relevant in the modelling of crystal growth. Here in general
$g\in H^{\frac12}(\partial\Omega)$, but for simplicity we assume that
$g \equiv \wD \in \R$ throughout this paper.

The anisotropic Allen--Cahn equation based on (\ref{eq:Eg}) 
is given by
\begin{subequations}
\begin{alignat}{2}
\epsilon\,
\frac{\partial u}{\partial t} & = 
\epsilon\,\nabla \,.\, A'(\nabla\, u) -\epsilon^{-1}\,\Psi'(u) 
\qquad && \mbox{in} \;\;\Omega_T\,, \label{eq:ACa} \\ 
\frac{\partial u}{\partial \vec\nu} & = 0\,, 
&& \mbox{on}\;\;\partial \Omega\times(0,T)\,, \label{eq:ACb} \\
\qquad u(\cdot,0) & = u_0 && \mbox{in} \;\;\Omega\,. \label{eq:ACc}
\end{alignat}
\end{subequations}
It was shown in \cite{ElliottS96,AlfaroGHMS10} that as $\epsilon\to0$ the 
zero level sets of $u$
converge to a sharp interface $\Gamma$ which moves by anisotropic mean 
curvature flow, i.e.\
\begin{equation} \label{eq:aMC}
\frac1{\gamma(\unitn)}\mathcal{V} =
\kappa_\gamma\,,
\end{equation}
where $\mathcal{V}$ is the velocity of $\Gamma$ in the direction of its normal
$\unitn$, and where $\kappa_\gamma$ is the anisotropic mean curvature
of $\Gamma$ with respect to the anisotropic surface energy
\begin{equation} \label{eq:EGamma}
\int_\Gamma \gamma(\unitn) \ds \,.
\end{equation}
In particular, $\kappa_\gamma$ is defined as the first variation of the 
above energy, which can be computed as 
$$
\kappa_\gamma := -\nabla_s\,.\,\gamma'(\vec{\unitn})\,,
$$
i.e.\ $\ddt \int_{\Gamma(t)} \gamma(\vec{\unitn}) \ds =
- \int_{\Gamma(t)} \kappa_\gamma\,\mathcal{V}\ds$;
where $\nabla_s\,.\,$ is the tangential divergence on $\Gamma$;
see e.g.\ \cite{CahnH74,TaylorCH92,Davis01}.

Similarly, with the help of formal asymptotics, see e.g.\
\cite{McFaddenWBCS93,WheelerM96,BellettiniP96,CaginalpE05}, 
it can be shown that the sharp
interface limit of (\ref{eq:CHa}--d) with $\theta = 1$ and $b(u) = b_0$,
with $b_0 \in \R_{>0}$,
is given by the following Mullins--Sekerka problem
\begin{subequations}
\begin{alignat}{2}
0 & = \Delta\, w && \qquad \mbox{in }\ 
\Omega_\pm\,,\label{eq:MSa} \\
b_0\, \left[\frac{\partial w}{\partial \unitn} \right]_-^+ & 
= - 2\,\mathcal{V}
&& \qquad \mbox{on }\ \Gamma(t)\,, \label{eq:MSb} \\
w & = \alpha\,\kappa_\gamma
&& \qquad \mbox{on }\ \Gamma(t)\,,
\label{eq:MSc} \\
\frac{\partial w}{\partial \vec\nu} & = 0 &&
\qquad \mbox{on } \partial \Omega\,,
\label{eq:MSd}
\end{alignat}
\end{subequations}
where $\Omega_\pm$ denote the domains occupied by the two phases,
$\Gamma = (\partial\Omega_+) \cap \Omega$ is the interface and
$[\cdot]_-^+$ denotes the jump across $\Gamma$ and with $\unitn$ pointing into
the set $\Omega_+$. 
Of course, if the natural boundary conditions 
(\ref{eq:CHc}) are changed to (\ref{eq:CHDbc}), then the limiting motion
becomes (\ref{eq:MSa}--c) together with
\begin{equation} \label{eq:wD}
w = \wD \qquad \mbox{on}\;\;\partial \Omega\,.
\end{equation}
The problem (\ref{eq:MSa}--c), (\ref{eq:wD})
models the supercooling of a molten pure substance, with $w$ playing the role
of a (rescaled) temperature. Then (\ref{eq:MSb}) is the so-called Stefan
condition, (\ref{eq:MSc}) is the anisotropic Gibbs--Thomson law without kinetic
undercooling and (\ref{eq:wD}) prescribes the supercooling at the boundary.
Here we note that in the quasi-static regime the heat diffusion was reduced to
Laplace's equation in $\Omega_\pm$ in (\ref{eq:MSa}).

On replacing heat diffusion with particle diffusion, the model 
(\ref{eq:MSa}--c), (\ref{eq:wD}), with $w$ now representing a (rescaled)
particle concentration, is relevant in isothermal crystal growth
where a density change occurs at the interfacer, see e.g.\
\cite{dendritic,crystal,jcg} and the references therein. We remark that in
order to recover the anisotropic Gibbs--Thomson law with kinetic
undercooling in place of (\ref{eq:MSc}) a viscous Cahn--Hilliard equation needs
to be considered, see e.g.\ \cite{BaiEGSS95}. We will look at this in more
detail in the forthcoming article \cite{vch}.

For later use we remark that a solution to (\ref{eq:MSa}--c), (\ref{eq:wD})
satisfies the energy identity
\begin{equation} \label{eq:Lyap}
\ddt\left(2\,\alpha\,\int_\Gamma \gamma(\unitn) \ds 
-2\,\wD\,\vol(\Omega_+)\right) + b_0\,\int_\Omega |\nabla\,w|^2 \dx = 0\,,
\end{equation}
see e.g.\ \cite{dendritic}, 
which is the sharp interface analogue of the corresponding formal
phase field energy bound 
\begin{equation} \label{eq:pfLyap}
\ddt\left(2\,\alpha\,\frac1\cPsi\,\mathcal{E}_\gamma(u) 
- \wD\,\int_\Omega u \dx \right) 
+ b_0\,\int_\Omega |\nabla\,w|^2 \dx \leq 0
\end{equation}
for the anisotropic Cahn--Hilliard equation (\ref{eq:CHa},b), (\ref{eq:CHDbc}) 
with $\theta=1$ and $b(u) = b_0$.

Lastly, the formal asymptotic limit of (\ref{eq:CHa}--d) with
$\theta=\epsilon$ and $b(u) = 1 - u^2$ is given by anisotropic surface
diffusion, i.e.\
\begin{equation} \label{eq:aSD}
\mathcal{V} = - \tfrac12\,\cPsi\,\alpha\,\Delta_s\,\kappa_\gamma\,,
\end{equation}
where $\Delta_s$ is the Laplace--Beltrami operator on $\Gamma$. The limit
(\ref{eq:aSD}) in the isotropic case (\ref{eq:iso}) was formally derived in 
\cite{CahnEN-C96}, and together with the techniques in e.g.\ 
\cite{WheelerM96,ElliottS96,GarckeNS98}
the anisotropic limit (\ref{eq:aSD}) is easily established.
More details on
the interpretation of anisotropic sharp interface motions as gradient flows for
(\ref{eq:EGamma}) and on their phase field equivalents can be found in
\cite{TaylorC94}.

It is the aim of this paper to introduce unconditionally stable finite element
approximations for the phase field models (\ref{eq:CHa}--d) and 
(\ref{eq:ACa}--c). Based on earlier work by the authors in the context of the
parametric approximation of anisotropic geometric evolution equations
\cite{triplejANI,ani3d}, the crucial idea here is to restrict the class of
anisotropies under consideration. 
The special structure of the chosen anisotropies
can then be exploited to develop discretizations that are stable without the
need for a regularization parameter and without a restriction on the time step 
size. 
In particular, the class of anisotropies that we will consider in this paper
is given by
\begin{equation} \label{eq:g1}
\gamma(\vec{p}) = \sum_{\ell=1}^L
\gamma_{\ell}(\vec{p}), \quad
\gamma_\ell(\vec{p}):= [{\vec{p}\,.\,G_{\ell}\,\vec{p}}]^\frac12\,,
\qquad \forall\ \vec p \in \R^d
\,,
\end{equation}
where $G_{\ell} \in \R^{d\times d}$, for $\ell=1\to L$, 
are symmetric and positive definite matrices. 
We note that (\ref{eq:g1}) corresponds to the special choice $r=1$ for the
class of anisotropies
\begin{equation} \label{eq:g}
\gamma(\vec{p}) = \left(\sum_{\ell=1}^L
[\gamma_{\ell}(\vec{p})]^r\right)^{\frac1r}
\qquad \forall\ \vec p \in \R^d 
\,,\qquad r \in [1,\infty)\,,
\end{equation}
which has been considered by the authors in \cite{ani3d,dendritic}. We remark
that anisotropies of the form (\ref{eq:g}) are always strictly convex norms.
In particular, they satisfy (\ref{eq:convex}); see Lemma~\ref{lem:1} below.
However, despite this seemingly restrictive choice, it is possible with
(\ref{eq:g}) to model and approximate a wide variety of anisotropies that are
relevant in materials science. For the sake of brevity, we refer to the
exemplary Wulff shapes in the authors' previous papers 
\cite{triplejANI,ani3d,clust3d,dendritic,ejam3d,crystal}. 
As we restrict ourselves to the class of anisotropies
(\ref{eq:g1}) in this paper, all of the numerical schemes introduced in 
Section~\ref{sec:3}, below, will feature only linear equations and linear 
variational inequalities.
The numerical approximation of anisotropic phase field models for the
class of anisotropies (\ref{eq:g}) is more involved, and we will
consider this in the forthcoming article \cite{vch}.

Let us shortly review previous work on the numerical analysis of
discretizations of anisotropic Allen--Cahn and Cahn--Hilliard 
phase field models. 
Fully explicit and nonlinear semi-implicit approximations of the Allen--Cahn
equation (\ref{eq:ACa}--c) are discussed in \cite[\S8]{DeckelnickDE05}.
In \cite{GraserKS11} several
time discretizations for (\ref{eq:ACa}--c) 
are considered, and
unconditional stability is shown for highly nonlinear, implicit
discretizations. Semi-implicit linearized discretizations are conditionally
stable on choosing a regularization parameter sufficiently large.
Moreover, numerical results for anisotropic Allen--Cahn equations have been
obtained in e.g.\ \cite{Paolini98,GarckeNS99,Benes03,Benes07}.
With particular reference to dendritic and crystal growth we mention
e.g.\ \cite{Kobayashi93,ElliottG96,KarmaR96,KarmaR98,DebierreKCG03}, 
where a forced anisotropic 
Allen--Cahn equation is coupled to a heat equation for the temperature.
We also mention the contributions of Christof Eck \cite{Eck04a,Eck04b,EckGS06},
who introduced homogenization methods into the field of crystal growth.

As far as we are aware, the presented paper includes the first numerical
analysis for an approximation of the anisotropic Cahn--Hilliard equation 
(\ref{eq:CHa},b). Finally, we mention that numerical computations for a 
generalized, sixth order Cahn--Hilliard equation, which is based on
a higher order regularization of the energy (\ref{eq:Eg}) in the case of a
non-convex anisotropy density function $\gamma$,
can be found in e.g.\ \cite{WiseKL07,LiLRV09}. 

The remainder of the paper is organized as follows. In Section~\ref{sec:2}
we consider a stable linearization of the gradient $A'$ for the anisotropy
function (\ref{eq:Ap}) and (\ref{eq:g1}). This will lay
the foundations for the stable finite element approximations introduced in
Section~\ref{sec:3}. 
Finally we present some numerical results in Section~\ref{sec:4}.

\section{Stable Linearization of $A'$} \label{sec:2}
The analysis in this paper is based on the special form (\ref{eq:g1}) of
$\gamma$. Note that for $\gamma$ satisfying (\ref{eq:g1}) it holds that
\begin{equation} \label{eq:Aprime}
A'(\vec p) = \gamma(\vec p)\,\gamma'(\vec p) 
\,, \qquad\text{where}\quad \gamma'(\vec p) = 
\sum_{\ell = 1}^L [\gamma_\ell(\vec p)]^{-1}\,G_\ell\,\vec p
\qquad \forall\ \vec p \in \R^d\setminus\{\vec 0\}\,.
\end{equation}

For later use we recall 
the elementary identity 
\begin{equation} \label{eq:element}
2\,r\,(r-s) = r^2 - s^2 + (r-s)^2\,.
\end{equation}

\begin{lem}\label{lem:1}
Let $\gamma$ be of the form (\ref{eq:g1}). Then $\gamma$ is convex and
the anisotropic operator $A$ satisfies
\begin{alignat}{2} \label{eq:mono}
 A'(\vec p) \,.\, (\vec p-\vec q) & \geq 
 \gamma(\vec p)\,[\gamma(\vec p)-\gamma(\vec q) ]
 \qquad && \forall\ \vec p \in \R^d\setminus\{\vec 0\}\,, \vec q \in \R^d\,, \\
 A(\vec p) & \leq \tfrac12\,\gamma(\vec q)\,
 \sum_{\ell=1}^L [\gamma_\ell(q)]^{-1}\,[\gamma_\ell(p)]^2 
 \qquad && \forall\ \vec p \in \R^d\,, \vec q \in \R^d\setminus\{\vec 0\}. 
 \label{eq:CS}
\end{alignat}
\end{lem}
\begin{proof}
We first prove (\ref{eq:gest}). It follows from (\ref{eq:Aprime}) and a 
Cauchy--Schwarz inequality that
\begin{align*}
\gamma'(\vec p) \,.\,\vec q = 
\sum_{\ell = 1}^L [\gamma_\ell(\vec p)]^{-1}\,(G_\ell\,\vec p)\,.\,\vec q
\leq \sum_{\ell = 1}^L \gamma_\ell(\vec q) = \gamma(\vec q)
\qquad \forall\ \vec p \in \R^d\setminus\{\vec 0\}\,, \vec q \in \R^d\,.
\end{align*}
Together with (\ref{eq:homo}) this implies (\ref{eq:convex}), i.e.\ $\gamma$ is
convex. Multiplying (\ref{eq:convex}) with $\gamma(\vec p)$ yields the desired 
result (\ref{eq:mono}).
Moreover, we have from a Cauchy--Schwarz inequality that 
\begin{equation*} 
\gamma(\vec p) = \sum_{\ell = 1}^L 
[\gamma_\ell(q)]^{\frac12}\,
\frac{\gamma_\ell(p)}{[\gamma_\ell(q)]^{\frac12}} \leq
[\gamma(\vec q)]^\frac12
\left( \sum_{\ell = 1}^L \frac{[\gamma_\ell(p)]^2}{\gamma_\ell(q)} 
\right)^{\frac12} 
\qquad \forall\ \vec p \in \R^d\,, \vec q \in \R^d\setminus\{\vec 0\}\,.
\end{equation*}
This immediately yields the desired result (\ref{eq:CS}), on recalling 
(\ref{eq:Ap}).
\end{proof}

Our aim now is to replace the highly nonlinear operator 
$A'(\vec p):\R^d\to\R^d$ in 
(\ref{eq:Aprime}) with a linearized approximation that still maintains the
crucial monotonicity property (\ref{eq:mono}). It turns out that the
natural linearization is already given in (\ref{eq:Aprime}). 
In particular, we let  
\begin{equation} \label{eq:B}
B(\vec q) := \begin{cases}
\gamma(\vec q)\,\displaystyle\sum_{\ell = 1}^L [\gamma_\ell(\vec q)]^{-1}\,G_\ell
& \vec q \not= \vec 0\,, \\
L\,\displaystyle\sum_{\ell = 1}^L G_\ell & \vec q = \vec 0\,.
\end{cases}
\end{equation}
Clearly it holds that
$$
B(\vec p)\,\vec p = A'(\vec p) 
\qquad \forall\ \vec p \in \R^d\setminus\{\vec 0\}\,,
$$
and it turns out that approximating $A'(\vec p)$ with $B(\vec q)\,\vec p$
maintains the monotonicity property (\ref{eq:mono}).

\begin{lem} \label{lem:B}
Let $\gamma$ be of the form (\ref{eq:g1}). Then it holds that
\begin{equation} \label{eq:Bmono}
[B(\vec q)\,\vec p]\,.\,(\vec p-\vec q) 
\geq \gamma(p)\,\left[\gamma(p)-\gamma(q)\right] 
\qquad \forall\ \vec p \,, \vec q \in \R^d \,.
\end{equation}
\end{lem}
\begin{proof}
Let $\vec p \in \R^d$.
If $\vec q \neq \vec 0$ it holds, on recalling (\ref{eq:CS}), that
\begin{align*}
[B(\vec q)\,\vec p]\,.\,(p-q) &= 
\gamma(q)\, \sum_{\ell=1}^L [\gamma_\ell(q)]^{-1}\,(p-q)\,.\,G_\ell\,p 
\geq \gamma(q) \,\sum_{\ell=1}^L 
 \gamma_\ell(p)\,([\gamma_\ell(q)]^{-1}\,\gamma_\ell(p) - 1) \\
& = \gamma(q) \,\sum_{\ell=1}^L [\gamma_\ell(q)]^{-1}\,[\gamma_\ell(p)]^2 
 - \gamma(q)\, \gamma(p) 
\geq \gamma(p)\,\left[\gamma(p)-\gamma(q)\right] .
\end{align*}
If $\vec q = \vec 0$, on the other hand, then it follows from a Cauchy--Schwarz
inequality that
\begin{align*}
[B(\vec q)\, \vec p]\,.\,(p-q) & = [B(\vec q)\,\vec p]\,.\,p = 
L\,\sum_{\ell=1}^L p \,.\,G_\ell\, p 
= L \sum_{\ell=1}^L [\gamma_\ell(\vec p)]^2 
\geq \left( \sum_{\ell=1}^L \gamma_\ell(\vec p)\right)^2 = [\gamma(p)]^2\,.
\end{align*}
\end{proof}

\begin{cor} \label{cor:B}
Let $\gamma$ be of the form (\ref{eq:g1}). Then it holds that
\begin{equation} \label{eq:Bstab}
[B(\vec q)\, \vec p]\,.\,(\vec p-\vec q) 
\geq A(\vec p) - A(q) 
\qquad \forall\ \vec p \,, \vec q \in \R^d \,.
\end{equation}
\end{cor}
\begin{proof}
The desired result follows immediately from Lemma~\ref{lem:B} on noting the
elementary identity (\ref{eq:element}). 
\end{proof}

\section{Finite Element Approximations} \label{sec:3}

Let $\{{\cal T}^h\}_{h>0}$ be a family of partitionings of $\Omega$ into
disjoint open simplices $\sigma$ with $h_{\sigma}:={\rm diam}(\sigma)$
and $h:=\max_{\sigma \in {\cal T}^h}h_{\sigma}$, so that
$\overline{\Omega}=\cup_{\sigma\in{\cal T}^h}\overline{\sigma}$.
Associated with ${\cal T}^h$ is the finite element space
\begin{equation*} \label{eq:Sh}
 S^h :=  \{\chi \in C(\overline{\Omega}) : \chi \mid_{\sigma} \mbox{ is linear }
 \forall\ \sigma \in {\cal T}^h\} \subset H^1(\Omega).
\end{equation*}
We introduce also
\begin{align*} \label{eq:K}
 K^h & := \{\chi \in S^h : |\chi| \leq 1 \mbox{ in } \Omega \}
 \subset \mathcal{K} 
 :=\{\eta \in H^1(\Omega) : |\eta| \leq 1 \mbox{ $a.e.$ in }\Omega\}\,.
\end{align*}
Let $J$ be the set of nodes of ${\cal T}^h$ and $\{p_{j}\}_{j \in J}$ the
coordinates of these nodes.
Let $\{\chi_{j}\}_{j\in J}$ be the standard basis
functions for $S^h$; that is $\chi_{j} \in S^h$ and $\chi_j(p_{i})=\delta_{ij}$
for all $i,j \in J$.
We introduce $\pi^h:C(\overline{\Omega})\rightarrow S^h$, the interpolation
operator, such that $(\pi^h \eta)(p_j)= \eta(p_j)$ for all $j \in J$. A
discrete semi-inner product on $C(\overline{\Omega})$ is then defined by
\begin{equation*} \label{eq:dip}
 (\eta_1,\eta_2)^h := \int_\Omega \pi^h(\eta_1(x)\,\eta_2(x))\dx 
\end{equation*}
with the induced discrete semi-norm given by
$|\eta|_h := [\,(\eta,\eta)^h\,]^{\frac{1}{2}}$, for
$\eta\in C(\overline{\Omega})$. Similarly, we denote the $L^2$--inner product
over $\Omega$ by $(\cdot,\cdot)$ with the corresponding norm given by
$|\cdot|_0$.

In addition to ${\cal T}^h$, let
$0= t_0 < t_1 < \ldots < t_{N-1} < t_N = T$ be a
partitioning of $[0,T]$ into possibly variable time steps $\tau_n := t_n -
t_{n-1}$, $n=1\rightarrow N$. We set
$\tau := \max_{n=1\rightarrow N}\tau_n$. 

We then consider the following fully practical, semi-implicit finite element
approximation for (\ref{eq:CHa}--d). 
For $n \geq 1$ find $(U^n,W^n) \in K^h \times S^h$ such that
\begin{subequations}
\begin{align}
& \theta \left(\displaystyle\frac{U^{n}-U^{n-1}}{\tau_n},
 \chi \right)^h + ( \pi^h[b(U^{n-1})]
 \,\nabla\, W^{n} , \nabla\, \chi )
 = 0 \qquad \forall\ \chi \in S^h, \label{eq:U} \\
& \epsilon \, (B(\nabla\, U^{n-1})\, \nabla\, U^{n}, \nabla\, [\chi -U^n])
 \geq 
 (\tfrac12\,\cPsi\,\alpha^{-1}\,W^{n} + \epsilon^{-1}\,U^{n-1},\chi - U^n)^h
 \qquad \forall\ \chi \in K^h\,, \label{eq:W}
\end{align}
\end{subequations}
where $U^0 \in K^h$ is an approximation of $u_0 \in {\cal K}$, e.g.\
$U^0 = {\pi}^h u_0$ for $u_0 \in C(\overline{\Omega})$.

Let 
\begin{equation} \label{eq:Eh}
\mathcal{E}_\gamma^h(U) = \tfrac12\,\epsilon\,|\gamma(\nabla\, U)|_0^2 +
\epsilon^{-1}\,(\Psi(U), 1)^h \qquad \forall\ U \in S^h
\end{equation}
be the natural discrete analogue of (\ref{eq:Eg}) 
and set $b_{\min} := \min_{s \in[-1,1]} b(s)$.

\begin{thm} \label{thm:stab}
There exists a solution $(U^n,W^n) \in K^h \times S^h$ to (\ref{eq:U},b)
with $(U^n,1) = (U^{n-1},1) = (U^0,1)$, and $U^n$ is unique. Moreover, it
holds that
\begin{equation} \label{eq:stab}
\mathcal{E}_\gamma^h(U^n) + \tau_n\,(2\,\theta\,\alpha)^{-1}\,\cPsi\,
( \pi^h[b(U^{n-1})]
 \,\nabla\, W^{n} , \nabla\, W^n ) \leq \mathcal{E}_\gamma^h(U^{n-1})\,.
\end{equation}
In addition, if $b_{\min} > 0$ and if 
$|(U^0, 1)| < \int_\Omega 1 \dx$ then $W^n$ is also unique. 
\end{thm}
\begin{proof}
The existence and uniqueness results follow straightforwardly with the 
techniques in \cite{BarrettBG99}, see also \cite{BloweyE92}, on
noting from (\ref{eq:B}) that $B(q) \in \R^{d\times d}$ is symmetric and
positive definite for all $\vec q \in \R^d$. 
Choosing $\chi = W^n$ in (\ref{eq:U}) and $\chi = U^{n-1}$ in (\ref{eq:W})
yields that
\begin{subequations}
\begin{align}
& \theta\, (U^{n}-U^{n-1}, W^n )^h + \tau_n\,( \pi^h[b(U^{n-1})]
 \,\nabla\, W^{n} , \nabla\, W^n )  = 0 \,, \label{eq:stabu} \\
& \epsilon \, (B(\nabla\, U^{n-1})\, \nabla\, U^{n}, \nabla\, [U^{n-1} -U^n])
 \geq 
 (\tfrac12\,\cPsi\,\alpha^{-1}\,W^n + \epsilon^{-1}\,U^{n-1}, U^{n-1} - U^n)^h
\,. 
\label{eq:stabw}
\end{align}
\end{subequations}
It follows from (\ref{eq:stabu},b), on recalling (\ref{eq:element})
and (\ref{eq:Bstab}), that
\begin{align*}
& \tfrac12\,\epsilon\,|\gamma(\nabla\, U^n)|_0^2 - \tfrac12\,
\epsilon^{-1}\,|U^n|_h^2 + \tau_n\,(2\,\theta\,\alpha)^{-1}\,\cPsi\,
( \pi^h[b(U^{n-1})] \,\nabla\, W^{n} , \nabla\, W^n ) 
\nonumber \\ & \hspace{9cm}
\leq \tfrac12\,\epsilon\,|\gamma(\nabla\, U^{n-1})|_0^2 - \tfrac12\,
\epsilon^{-1}\,|U^{n-1}|_h^2 \,. 
\end{align*}
This yields the desired result (\ref{eq:stab}) on adding the constant
$\frac12\,\epsilon^{-1}\,\int_\Omega 1 \dx$ on both sides.
\end{proof}

\begin{rem}
On replacing (\ref{eq:U}) with
\begin{equation} \label{eq:U2}
 \epsilon \left(\displaystyle\frac{U^{n}-U^{n-1}}{\tau_n},
 \chi \right)^h + \tfrac12\,\cPsi\,\alpha^{-1}\,( W^{n} , \chi )^h 
 = 0 \qquad \forall\ \chi \in S^h
\end{equation}
we obtain a finite element approximation for (\ref{eq:ACa}--c).
Similarly to Theorem~\ref{thm:stab} existence of a unique solution 
$(U^n,W^n) \in K^h \times S^h$ to (\ref{eq:U2}), (\ref{eq:W}), which is
unconditionally stable, can then be shown. In particular, the solution
$(U^n,W^n)$ to (\ref{eq:U2}), (\ref{eq:W}) satisfies the bound (\ref{eq:stab}) 
with the second term on the left hand side of (\ref{eq:stab}) replaced by 
$\tau_n\,\epsilon^{-1}\,(\frac12\,\cPsi\,\alpha^{-1})^2\,|W^n|_h^2$.
\end{rem}

\begin{rem} \label{rem:implicit}
On replacing the term $\epsilon^{-1}\,U^{n-1}$ on the right hand side of
(\ref{eq:W}) with $\epsilon^{-1}\,U^{n}$, we obtain an implicit scheme for
which the existence of a unique solution $U^n$ can only be shown if
the time step $\tau_n$ satisfies a very severe constraint of the form
$\tau_n < C\,\epsilon^3\,\theta\,\alpha^{-1}$, 
where the constant $C>0$ depends only on the anisotropy $\gamma$ and on the
mobility $b$. 
In the isotropic case (\ref{eq:iso}) with constant mobility coefficient 
$b(u) = b_0 \in \R_{>0}$ this constraint can be made precise and is given by
\begin{equation} \label{eq:tau}
\tau_n < 2\,\cPsi\,\epsilon^3\,\theta\,(\alpha\,b_{0})^{-1};
\end{equation}
see e.g.\ \cite{BloweyE92}.
\end{rem}

In the remainder of this section we consider the numerical approximation of
(\ref{eq:MSa}--c), (\ref{eq:wD}). In particular, we introduce a 
finite element approximation for (\ref{eq:CHa}--c), (\ref{eq:CHDbc}). To this
end, let 
\begin{equation} \label{eq:Sh0}
 S^h_0:= \{\chi \in S^h : \chi = 0 \ \mbox{ on $\partial\Omega$} \}  
\quad \mbox{and} \quad
 \ShD:= \{\chi \in S^h : \chi = \wD\ 
\mbox{ on $\partial\Omega$} \} 
\,.
\end{equation}
We then consider the following fully practical, semi-implicit finite element
approximation for (\ref{eq:CHa}--c), (\ref{eq:CHDbc}) with $\theta=1$ and
$b(u)=b_0>0$.
For $n \geq 1$ find $(U^n,W^n) \in K^h \times S^h_D$ such that
\begin{subequations}
\begin{align}
& \left(\displaystyle\frac{U^{n}-U^{n-1}}{\tau_n},
 \chi \right)^h + b_0\,( \nabla\, W^{n} , \nabla\, \chi )
 = 0 \qquad \forall\ \chi \in S^h_0, \label{eq:Ubc} \\
& \epsilon \, (B(\nabla\, U^{n-1})\, \nabla\, U^{n}, \nabla\, [\chi -U^n])
 \geq 
 (\tfrac12\,\cPsi\,\alpha^{-1}\,W^{n} + \epsilon^{-1}\,U^{n-1},\chi - U^n)^h
 \qquad \forall\ \chi \in K^h\,. \label{eq:Wbc}
\end{align}
\end{subequations}

Let 
\begin{equation} \label{eq:Ehbc}
\mathcal{F}_{\gamma}^h(U) = 
2\,\alpha\,\frac1\cPsi\,\mathcal{E}_{\gamma}^h(U) - \wD\,(U,1)
\qquad \forall\ U \in S^h \,.
\end{equation}

Then it holds that the solution to (\ref{eq:Ubc},b) satisfies a discrete
analogue to (\ref{eq:pfLyap}). 

\begin{thm} \label{thm:stabbc}
There exists a unique solution 
$(U^n,W^n) \in K^h \times S^h_D$ to (\ref{eq:Ubc},b).
Moreover, it holds that
\begin{equation} \label{eq:stabbc}
\mathcal{F}_{\gamma}^h(U^n) 
+ \tau_n\,b_0\,|\nabla\, W^{n}|_0^2 
\leq \mathcal{F}_{\gamma}^h(U^{n-1})\,.
\end{equation}
\end{thm}
\begin{proof}
The existence and uniqueness proof is similar to the proof of
Theorem~\ref{thm:stab}, but we detail it here for the readers' convenience.
Let $\mathcal{G}^h : S^h \to S^h_0$ denote the discrete solution operator for
the homogeneous Dirichlet problem on $\Omega$, i.e.\
\begin{equation} \label{eq:Gh}
(\nabla\,[\mathcal{G}^h\,v^h], \nabla\,\chi) = (v^h, \chi)^h\qquad \forall\ \chi
\in S^h_0\,,\quad \forall\ v^h \in S^h\,.
\end{equation}
Hence for $U^n\in K^h$ we have that (\ref{eq:Ubc}) is equivalent to
\begin{equation} \label{eq:GU}
W^n = \wD - b_0^{-1}\,\mathcal{G}^h[ \tfrac{U^{n}-U^{n-1}}{\tau_n}]\,.
\end{equation}
It follows from (\ref{eq:Wbc}) and (\ref{eq:GU}) that $U^n \in K^h$ is such 
that
\begin{align} 
& 
\epsilon\,(B(\nabla\, U^{n-1})\,\nabla\, U^n,\nabla\, (\chi-U^n)) + 
(\hatalpha\,b_0)^{-1}\,(\mathcal{G}^h[ \tfrac{U^{n}-U^{n-1}}{\tau_n}],
\chi-U^n)^h \nonumber \\ & \hspace{6.5cm}
\geq (\hatalpha^{-1}\,\wD + \epsilon^{-1}\,U^{n-1},\chi-U^n)^h 
\qquad \forall\ \chi \in K^h\,,
\label{eq:EL}
\end{align}
where $\hatalpha:= \frac2\cPsi\,\alpha > 0$.
There exists a unique $U^n \in K^h$
solving (\ref{eq:EL}) since this is the 
Euler--Lagrange variational inequality of the strictly
convex minimization problem
\begin{equation*}
\min_{z^h \in K^h} \!
\left\{\tfrac\epsilon{2}\,(B(\nabla\,U^{n-1})\,\nabla\, z^h, \nabla\, z^h) 
+(2\,\tau_n\,\hatalpha\,b_0)^{-1}\,|\nabla\,{\cal G}^h
(z^h-U^{n-1})|_0^2 
- (\hatalpha^{-1}\,\wD+\epsilon^{-1}\,U^{n-1}, z^h)^h\right\} . 
\end{equation*}
Therefore, on recalling (\ref{eq:GU}), we have existence of a unique
solution $(U^n,W^n) \in K^h\times S^h_D$ to
(\ref{eq:Ubc},b).
Choosing $\chi = W^n - \wD$ in (\ref{eq:Ubc}) and $\chi = U^{n-1}$ in 
(\ref{eq:Wbc}) yields that
\begin{align*}
& (U^{n}-U^{n-1}, W^n - \wD )^h + 
 \tau_n\,b_0\,(\nabla\, W^{n} , \nabla\, W^n )  = 0 \,, \\
& \epsilon \, (B(\nabla\, U^{n-1})\, \nabla\, U^{n}, \nabla\, [U^{n-1} -U^n])
 \geq (\hatalpha^{-1}\,W^{n} + \epsilon^{-1}\,U^{n-1}, U^{n-1} - U^n)^h\,. 
\end{align*}
Hence the desired result (\ref{eq:stabbc}) follows from 
(\ref{eq:element}) and (\ref{eq:Bstab}).
\end{proof}

\begin{rem} \label{rem:bl}
It is easy to show that for $U^{n-1} = 1$ and
\begin{equation} \label{eq:bl}
-\alpha^{-1}\, \wD \leq \frac2\cPsi\,\epsilon^{-1}
\end{equation}
the unique solution to (\ref{eq:Ubc},b) is given by $U^n = 1$ and $W^n = \wD$.
However, if the phase field parameter $\epsilon$ does not satisfy 
(\ref{eq:bl}), then $U^n = 1$ and $W^n = \wD$ is no longer the 
solution to (\ref{eq:Ubc},b). In fact, in practice
it is observed that for $\epsilon$ sufficiently large the solution $U^n$
exhibits a boundary layer close to $\partial\Omega$ where $U^n < 1$. This
artificial boundary layer, which formally can be shown to be also admitted by
the continuous problem (\ref{eq:CHa},b,d), (\ref{eq:CHDbc}), 
is an undesired effect of the phase field approximation for the sharp 
interface problem (\ref{eq:MSa}--c), (\ref{eq:wD}).
\end{rem}

\section{Numerical Experiments} \label{sec:4}

In this section we report on numerical experiments for the proposed finite
element approximations. For the implementation of the approximations we have
used the adaptive finite element toolbox ALBERTA, see \cite{Alberta}. We employ
the adaptive mesh strategy introduced in \cite{voids} and
\cite{voids3d}, respectively, for $d=2$ and $d=3$.
This results in a fine mesh of uniform mesh size 
$h_f$ inside the interfacial region $|U^{n-1}|<1$ and a coarse mesh of uniform 
mesh size $h_c$ further away from it. 
Here $h_{f} = \frac{2\,H}{N_{f}}$ and $h_{c} =  \frac{2\,H}{N_{c}}$
are given by two integer numbers $N_f >  N_c$, where we assume from now on that
$\Omega = (-H,H)^d$.
As a solution method for the resulting system of algebraic equations 
we use the Uzawa-multigrid iteration from \cite{voids3d},
which is based on the ideas in \cite{GraserK07}. 
We remark that
recently various alternative solution methods have been proposed, see e.g.\
\cite{mgch,BlankBG11,HintermullerHT11,GraserKS12}.

For all the computations we take $H=\frac12$, 
unless otherwise stated.
Throughout this section the initial data $u_0\in
C(\overline\Omega)$ is chosen with a well developed interface of width
$\epsilon\,\pi$, in which $u_0$ varies smoothly. 
Details of such initial data can be found in e.g.\ 
\cite{voids,wccmproc,voids3d}.
Unless otherwise stated we always set 
$\epsilon^{-1}=16\,\pi$ and $N_f = 128$, $N_c=16$. In addition, we employ
uniform time steps $\tau_n = \tau$, $n = 1 \to N$.

For numerical approximations of (\ref{eq:aMC}) we employ the scheme
(\ref{eq:U2}), (\ref{eq:W}).
In computations for (\ref{eq:aSD}) we use the scheme (\ref{eq:U},b) and
fix $b(u) = 1 - u^2$, $\theta = \epsilon$ and 
$\alpha = \frac2\cPsi$. In all other cases, i.e.\ for
the sharp interface limits (\ref{eq:MSa}--c) with (\ref{eq:MSd}) or 
(\ref{eq:wD}), we fix $b(u) = b_0 = 2$, $\theta = 1$ and $\alpha=1$ 
unless otherwise stated.

For the anisotropies in our numerical results we always choose among
\begin{alignat*}{2}
& \text{\sc ani$_1^{(\delta)}$:} \qquad
\gamma_1(\vec{p}) = \sum_{j=1}^d\, \left[ \delta^2\,
|\vec{p}|^2+ p_j^2\,(1-\delta^2) \right]^{\frac12} \,,\quad
&& \text{with}\quad \delta > 0\,, 
\\
& \text{{\sc ani$_2$:} $\gamma$ as on the bottom of Figure~3 in \cite{ani3d}},
\qquad\qquad
&& \text{{\sc ani$_3$:} $\gamma$ as on the right of Figure~2 in 
\cite{dendritic}}, \\
& \text{{\sc ani$_4$:} $\gamma$ as in Figure~3 in \cite{jcg}}.
\end{alignat*}
We remark that {\sc ani$_1^{(\delta)}$} is a regularized $l_1$--norm, so that
its Wulff shape for $\delta$ small is given by a smoothed square (in 2d) or a
smoothed cube (in 3d) with nearly flat sides/facets. Anisotropies with such 
flat sides or facets are called crystalline. Also the choices
{\sc ani$_i$}, $i=2\to4$, represent nearly crystalline anisotropies. Here the
Wulff shapes are given by a smoothed cylinder, a smoothed hexagon and a
smoothed hexagonal prism, respectively. Finally, we denote by
{\sc ani$_1^{\star}$} the anisotropy
{\sc ani$_1^{(0.01)}$} rotated by $\tfrac\pi4$ in the $x_1-x_2$-plane. 

\subsection{Numerical results in 2d}

A numerical experiment for (\ref{eq:aMC}) with the help of  
the approximation (\ref{eq:U2}), (\ref{eq:W}) for
the Allen--Cahn equation (\ref{eq:ACa}--c) can be seen in Figure~\ref{fig:aMC}.
Here the initial profile is given by a circle with radius $0.3$. We set
$\tau=10^{-4}$ and $T=0.05$. As expected, the round interface first becomes
facetted, before is shrinks to a point and disappears.
\begin{figure}
\center
\includegraphics[angle=-0,width=0.19\textwidth]{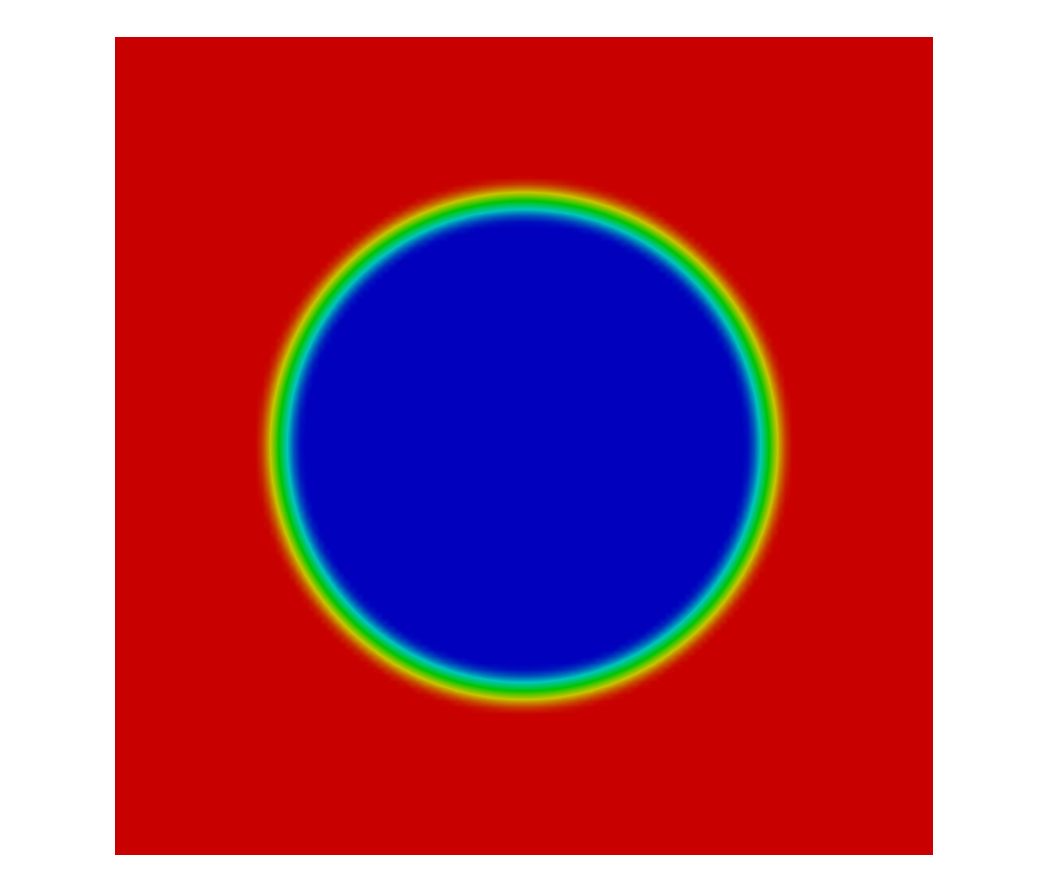}
\includegraphics[angle=-0,width=0.19\textwidth]{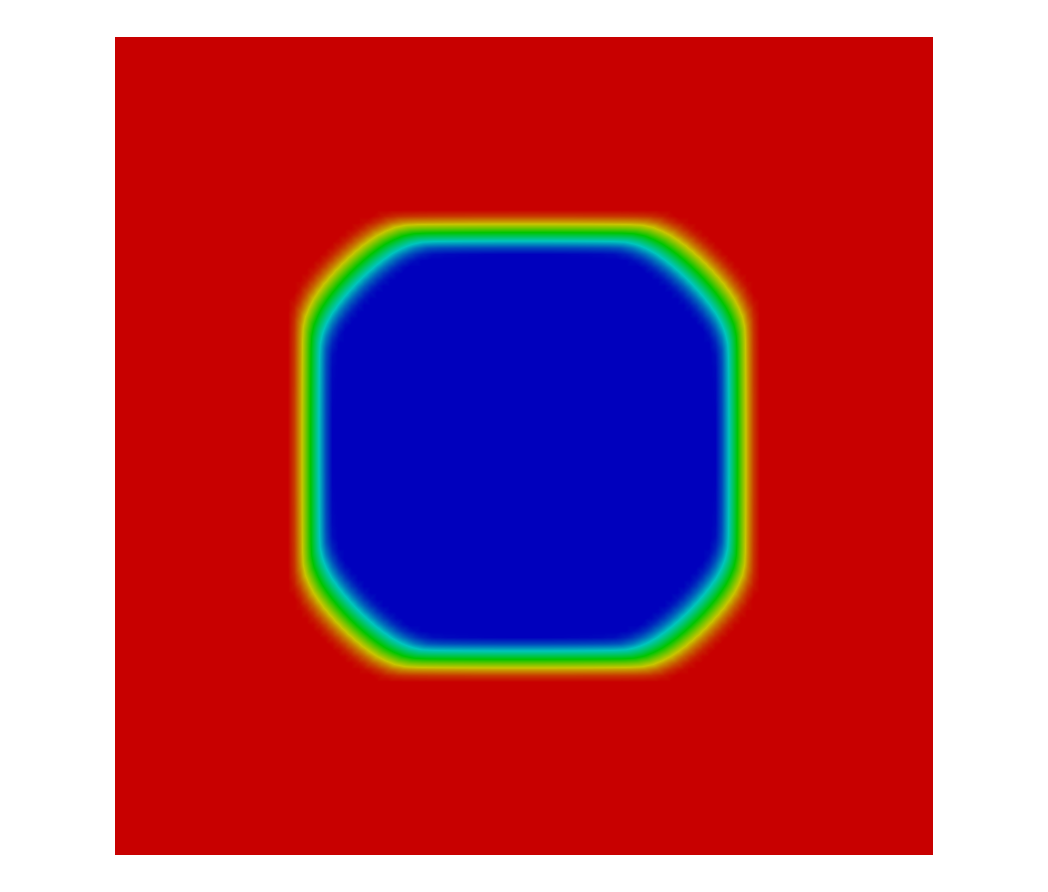}
\includegraphics[angle=-0,width=0.19\textwidth]{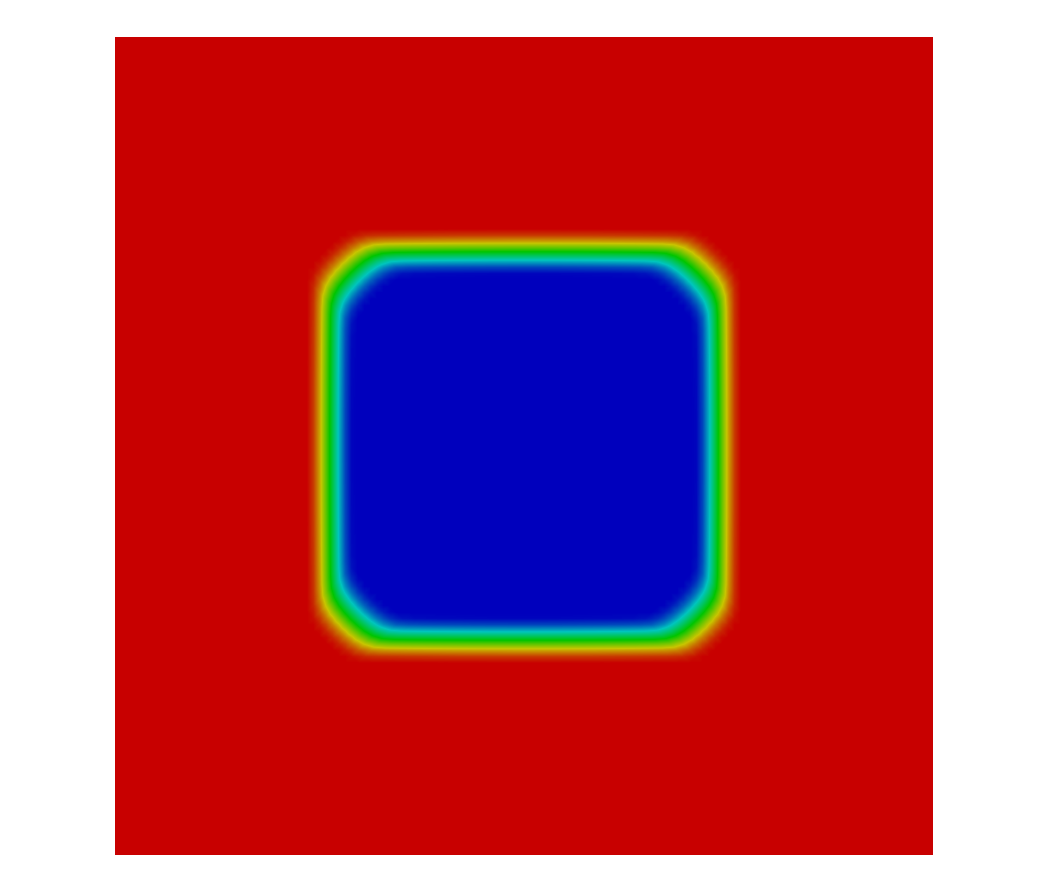}
\includegraphics[angle=-0,width=0.19\textwidth]{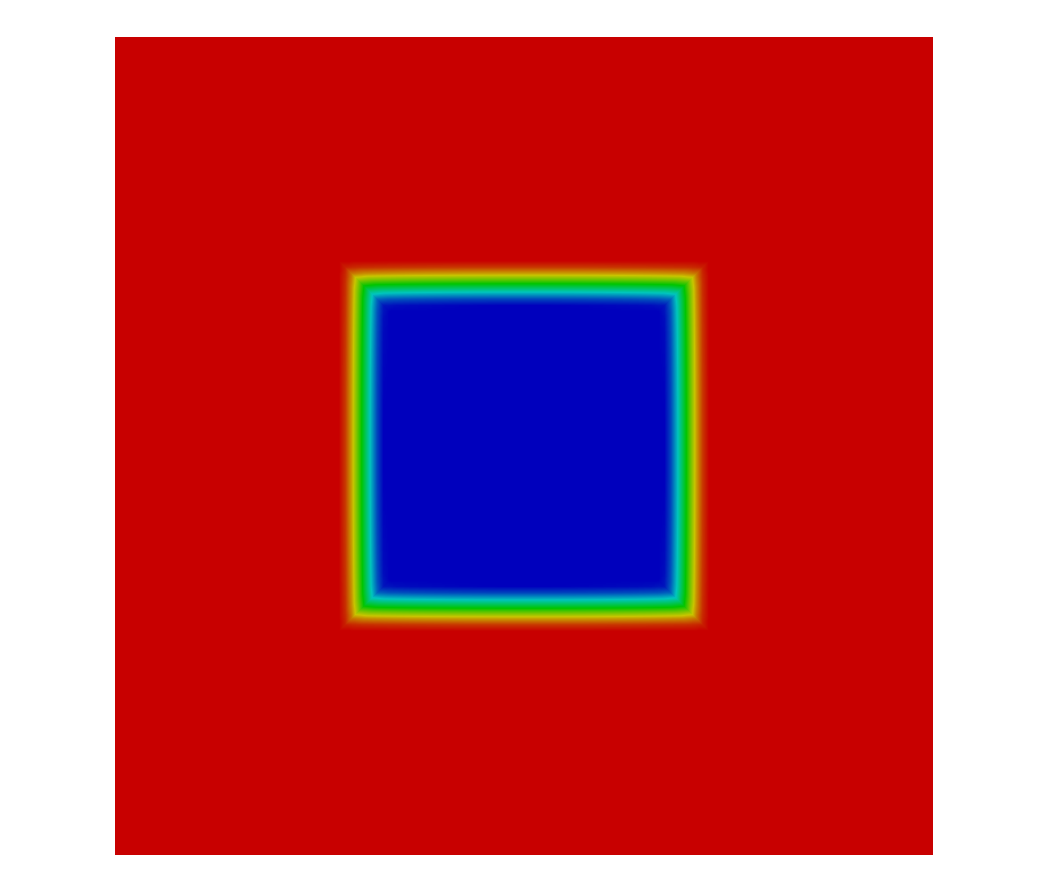}
\includegraphics[angle=-0,width=0.19\textwidth]{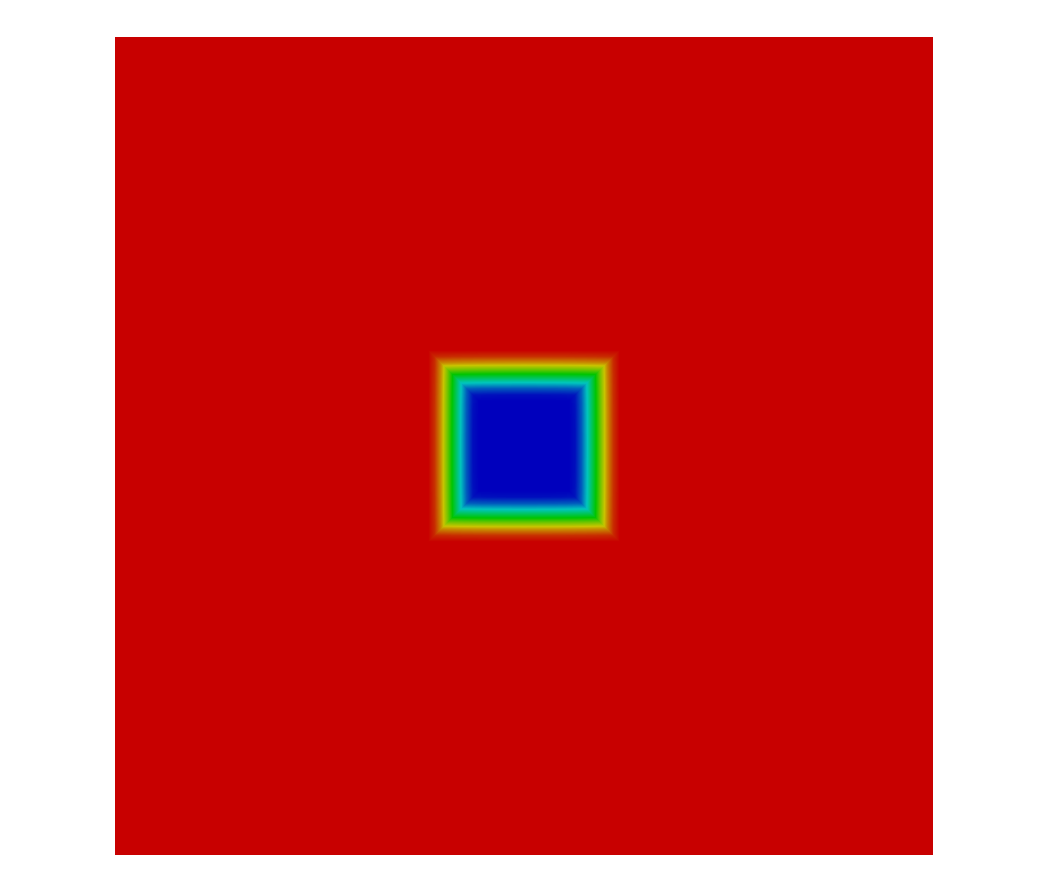}
\includegraphics[angle=-90,width=0.35\textwidth]{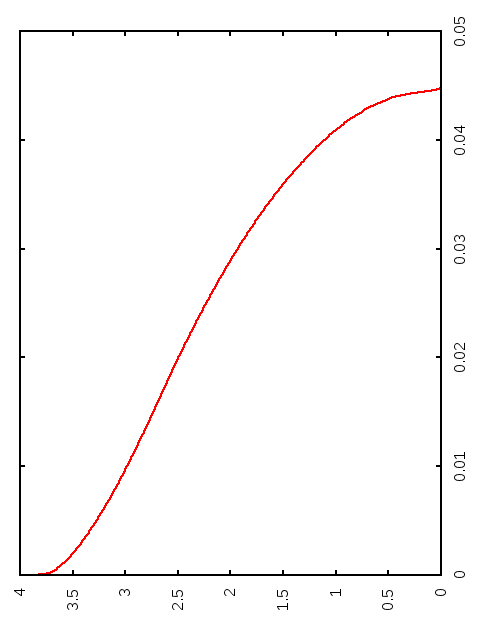}
\caption{({\sc ani$_1^{(0.01)}$}) 
A phase field approximation for the anisotropic mean curvature flow
(\ref{eq:aMC}).
Snapshots of the solution at times $t=0,\,5\times10^{-3},\,10^{-2},\,
2\times10^{-2},\,4\times10^{-2}$.
A plot of $\mathcal{E}_\gamma^h$ below.
}
\label{fig:aMC}
\end{figure}%

A numerical experiment for (\ref{eq:aSD}) with the help of
the approximation (\ref{eq:U},b), for
the Cahn--Hilliard equation (\ref{eq:CHa}--d) 
can be seen in Figure~\ref{fig:aSD}.
Here the initial profile is given by two circles with radii $0.2$ and $0.15$. 
We set $\tau=10^{-6}$ and $T=10^{-4}$. We observe that the two connected 
components of the inner phase each take on the form of the hexagonal Wulff 
shape.
\begin{figure}
\center
\includegraphics[angle=-0,width=0.19\textwidth]{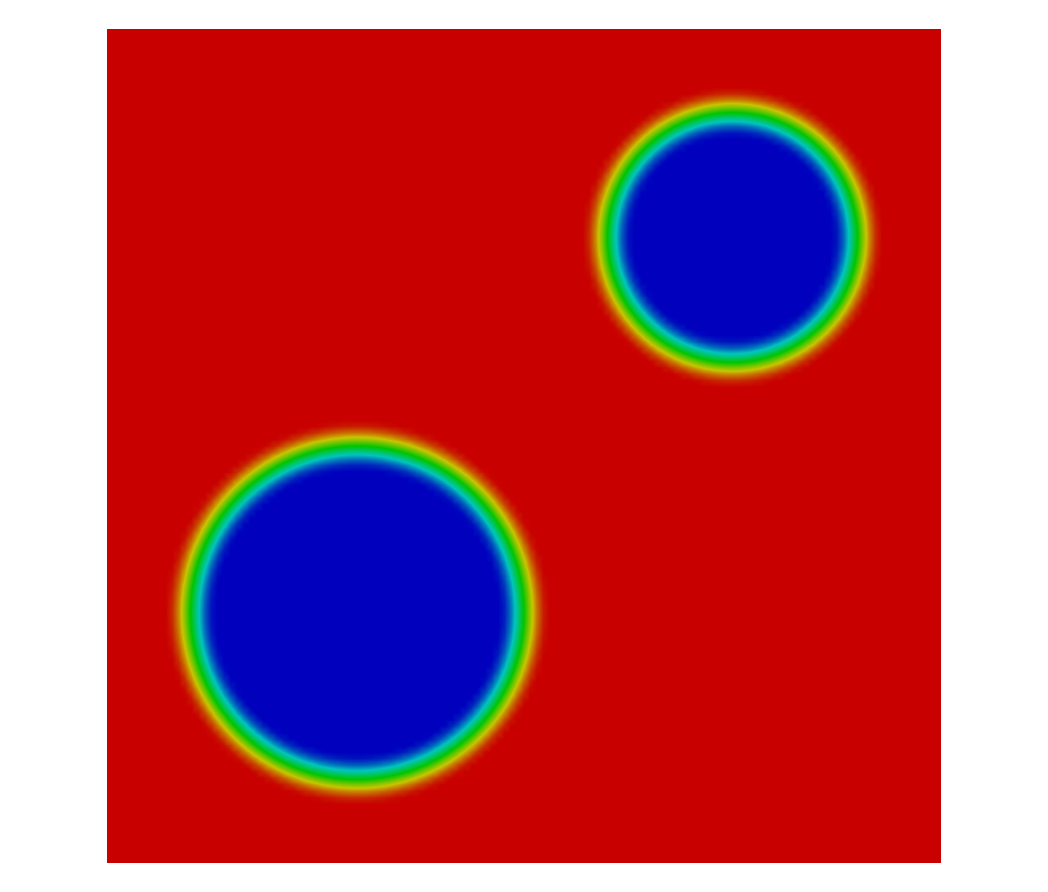}
\includegraphics[angle=-0,width=0.19\textwidth]{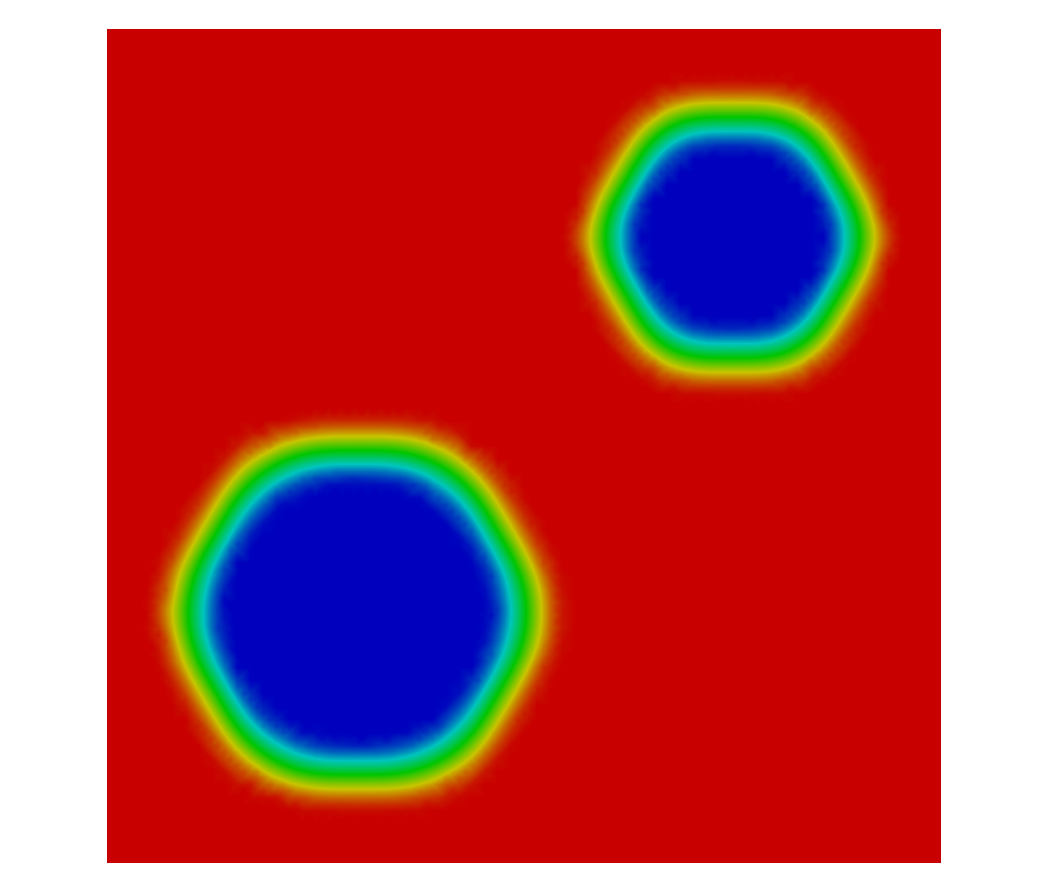}
\includegraphics[angle=-0,width=0.19\textwidth]{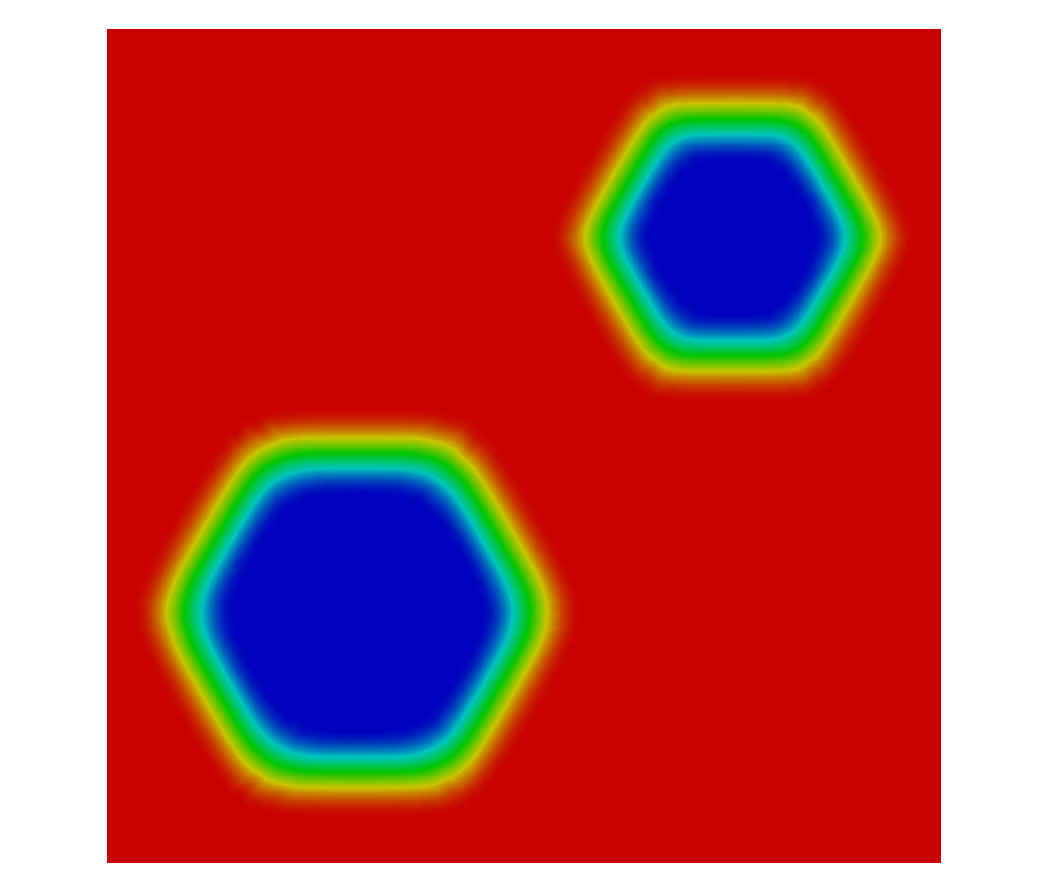}
\includegraphics[angle=-0,width=0.19\textwidth]{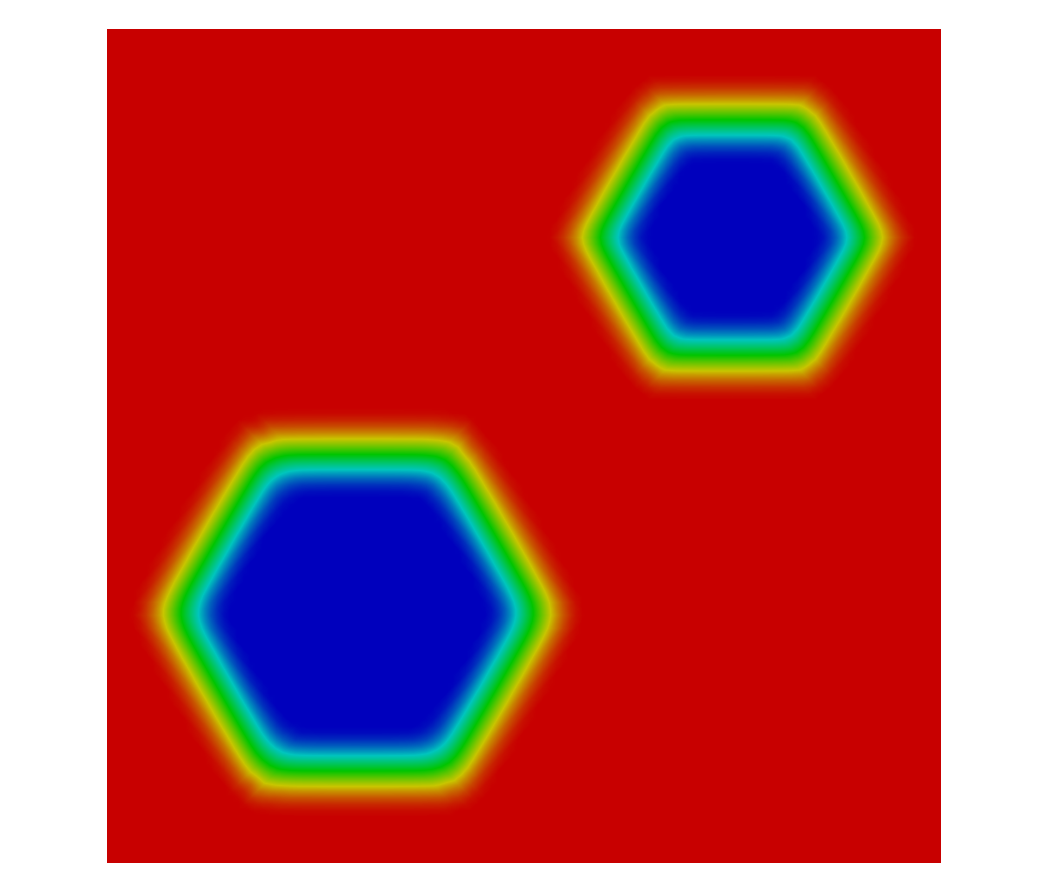}
\includegraphics[angle=-0,width=0.19\textwidth]{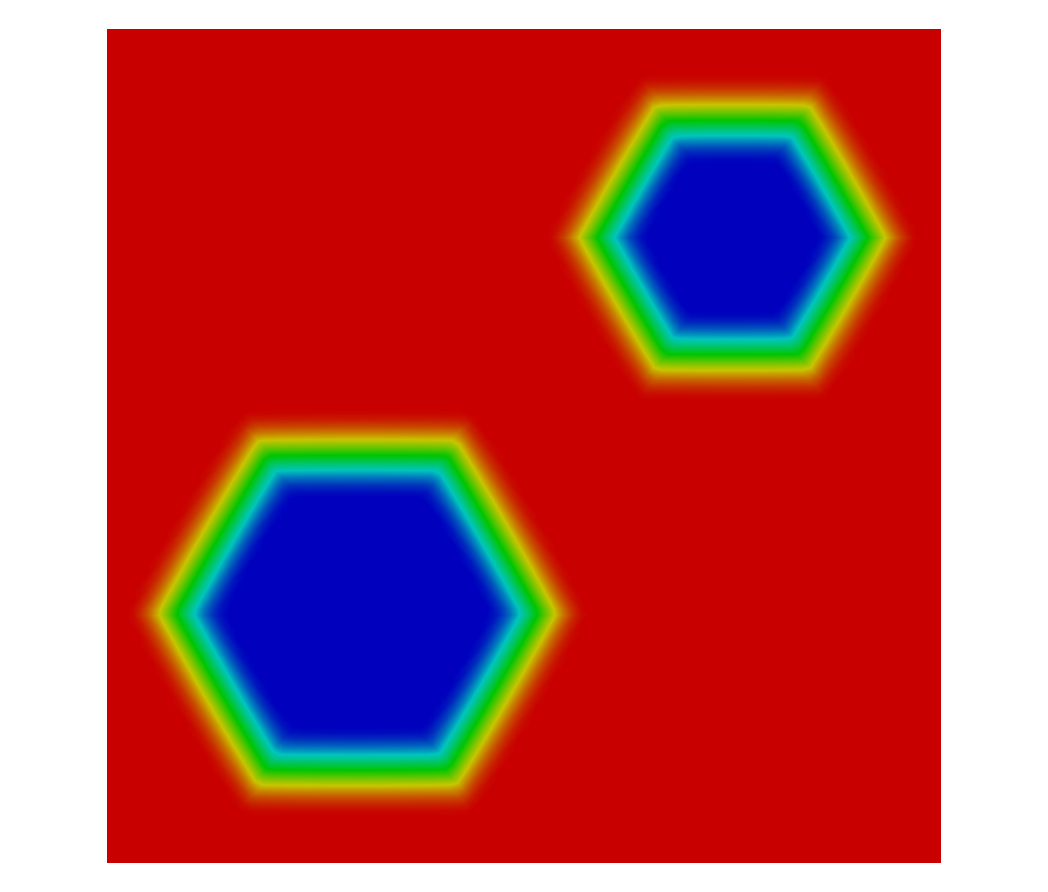}
\includegraphics[angle=-90,width=0.35\textwidth]{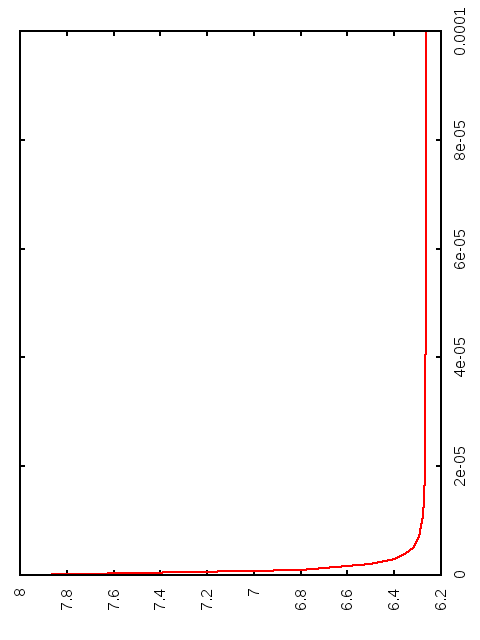}
\caption{({\sc ani$_3$})
A phase field approximation for anisotropic surface diffusion (\ref{eq:aSD}).
Snapshots of the solution at times $t=0,\,2\times10^{-6},\,
5\times10^{-6},\,10^{-5},\,10^{-4}$.
A plot of $\mathcal{E}_\gamma^h$ below.
}
\label{fig:aSD}
\end{figure}%

A repeat of the experiment but now for $b(u)=b_0=2$, so that the sharp 
interface limit is given by the Mullins--Sekerka problem (\ref{eq:MSa}--d), 
is shown in Figure~\ref{fig:aMS}. Here we set $\tau=10^{-5}$ and 
$T=5\times10^{-3}$. Now, in contrast to the evolution in Figure~\ref{fig:aSD},
the smaller region shrinks so that eventually there is only one connected
component of the inner phase. Of course, the final interface is converging to
the hexagonal Wulff shape.
\begin{figure}
\center
\includegraphics[angle=-0,width=0.19\textwidth]{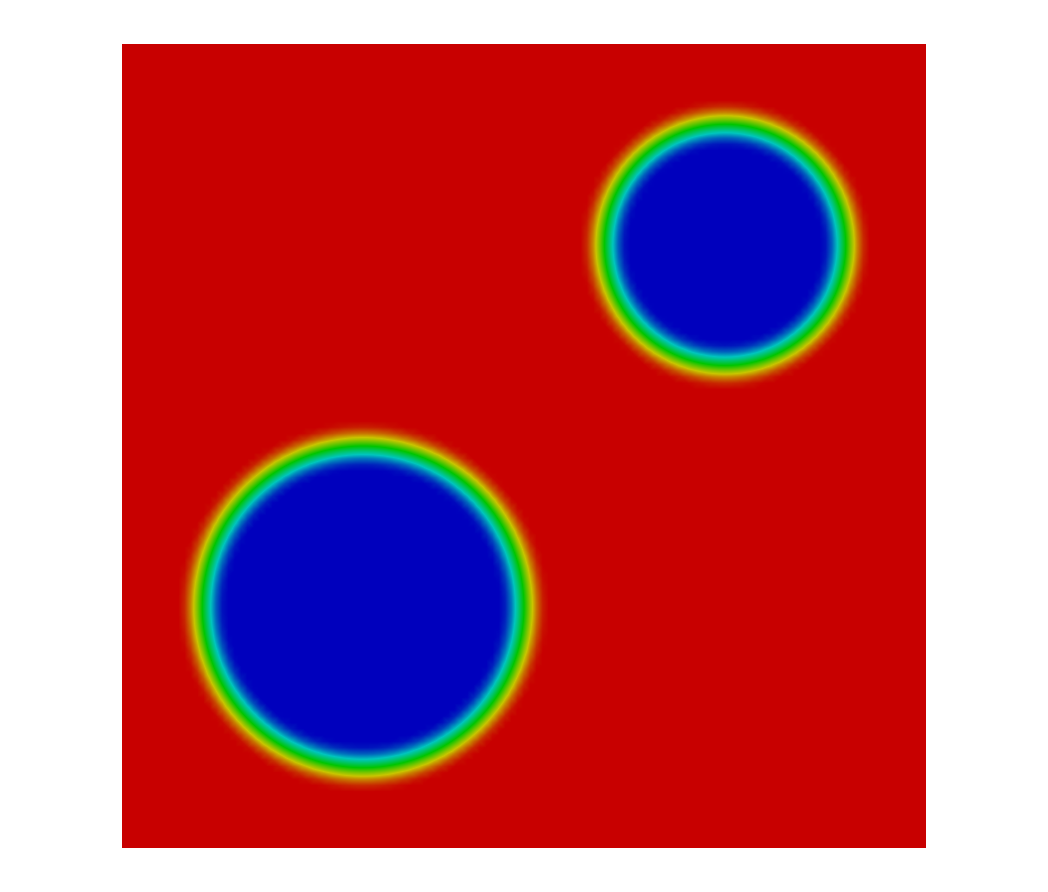}
\includegraphics[angle=-0,width=0.19\textwidth]{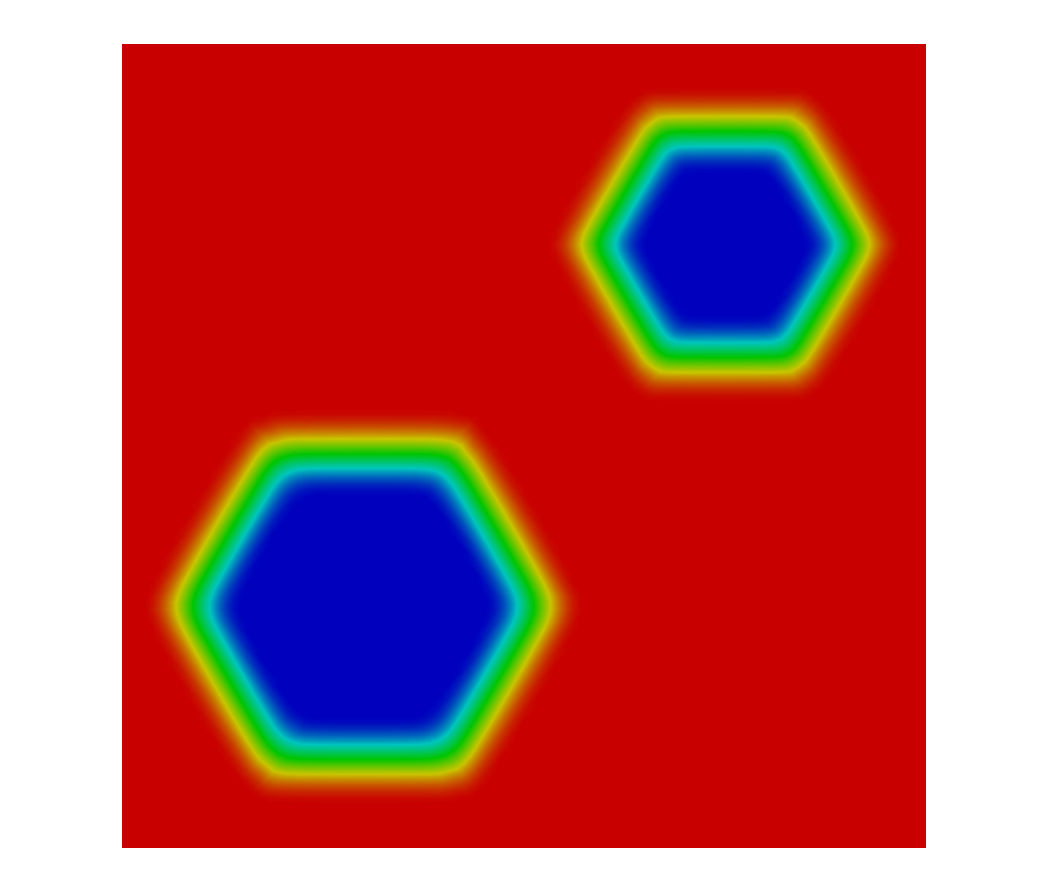}
\includegraphics[angle=-0,width=0.19\textwidth]{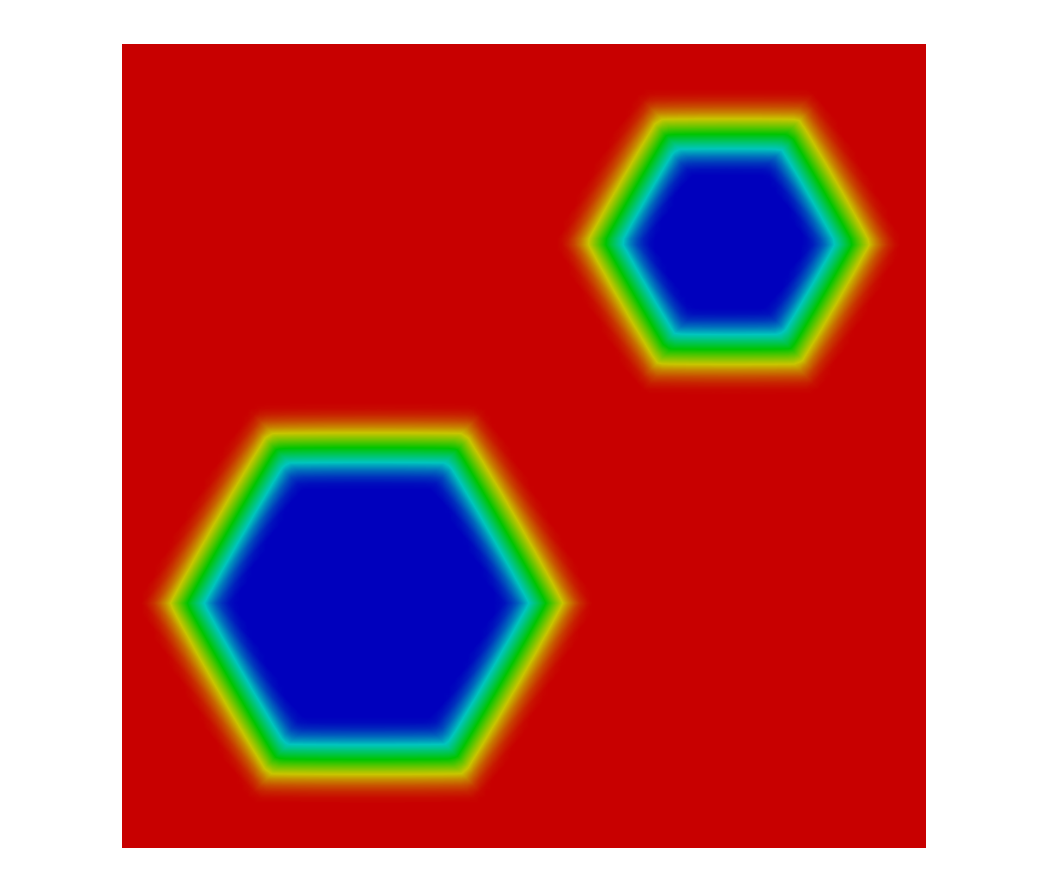} 
\includegraphics[angle=-0,width=0.19\textwidth]{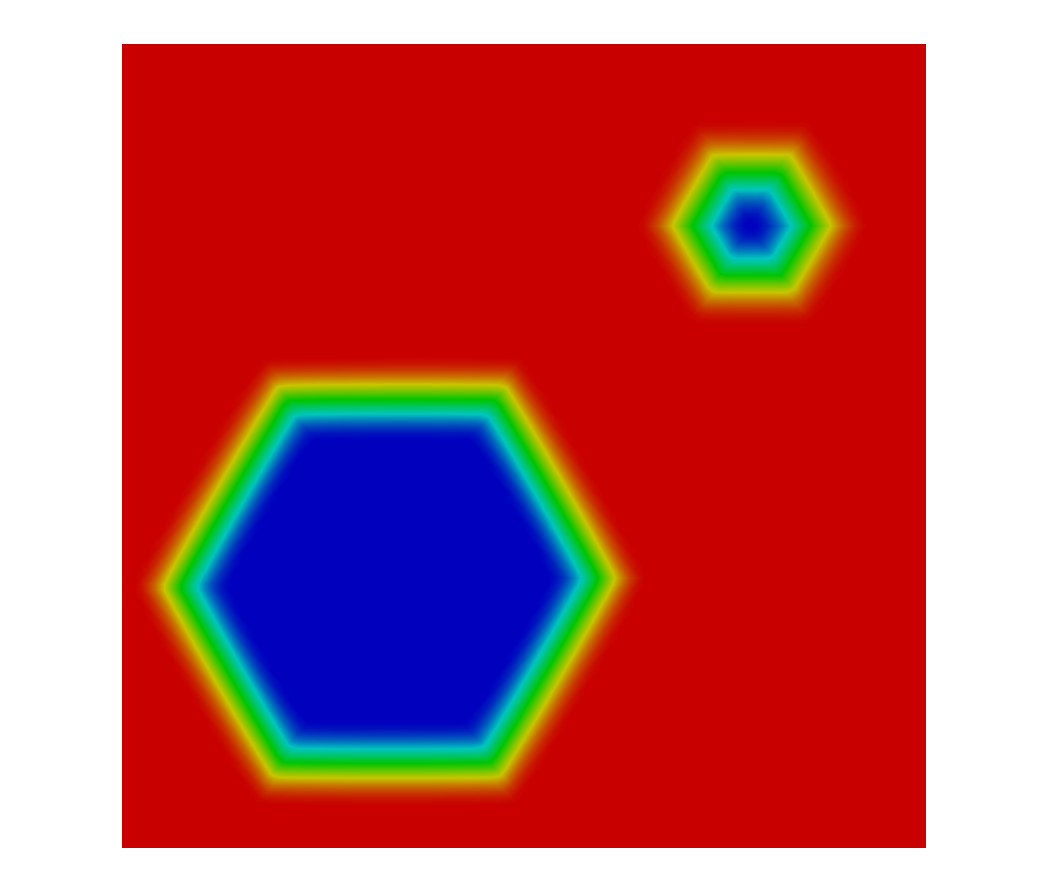}
\includegraphics[angle=-0,width=0.19\textwidth]{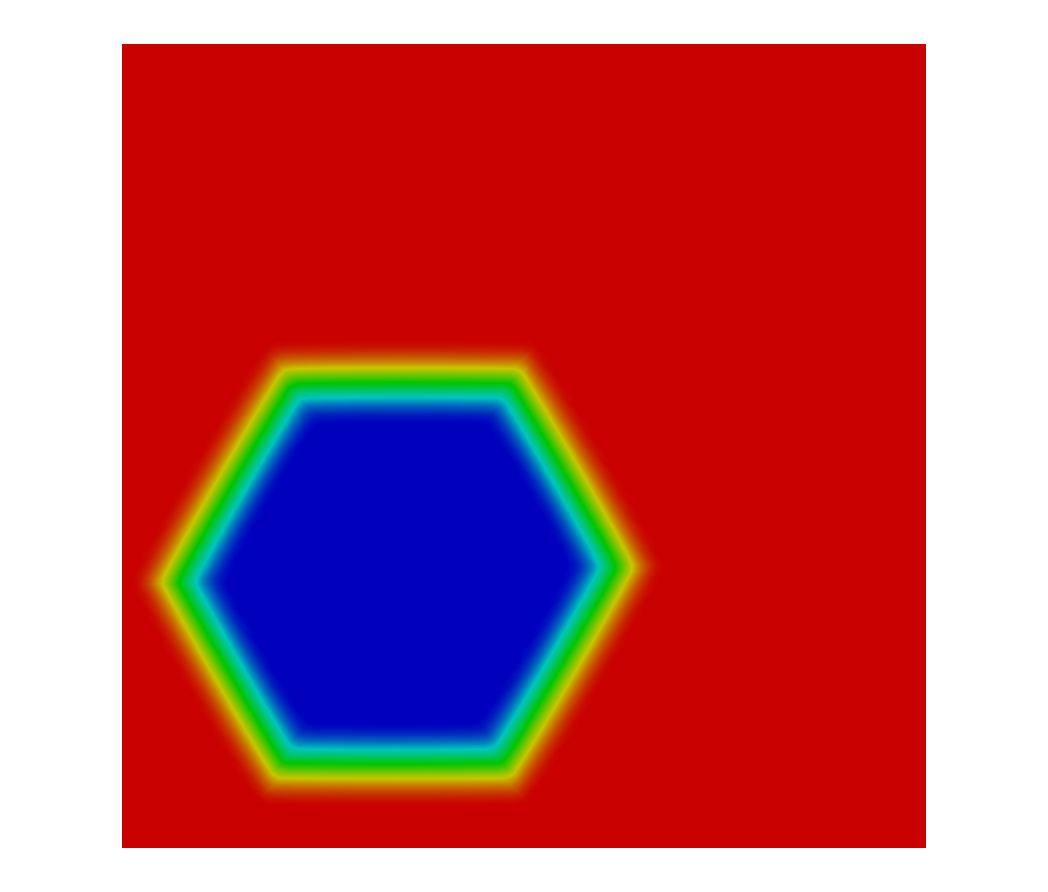}
\includegraphics[angle=-90,width=0.35\textwidth]{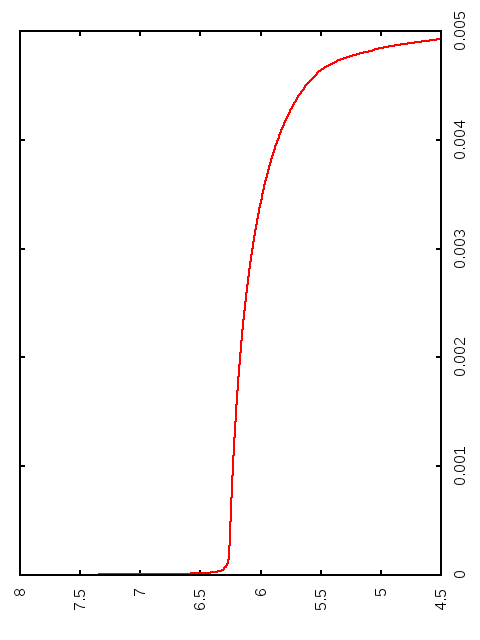}
\caption{({\sc ani$_3$})
A phase field approximation for the anisotropic Mullins--Sekerka problem
(\ref{eq:MSa}--d).
Snapshots of the solution at times $t=0,\,10^{-4},\,
10^{-3},\,4.6\times10^{-3},\,5\times10^{-3}$.
A plot of $\mathcal{E}_\gamma^h$ below.
}
\label{fig:aMS}
\end{figure}%

The remaining computations in this subsection are for the scheme 
(\ref{eq:Ubc},b).
In order to visualize the possible onset of a boundary layer as explained in
Remark~\ref{rem:bl}, we present a computation for (\ref{eq:Ubc},b) with the
initial data $U^0 = u_0 = 1$. 
As we set $\alpha = 1$, the critical 
value for $\wD$ in (\ref{eq:bl}) is 
$-\frac2\cPsi\,\epsilon^{-1} = -\frac4\pi\,16\,\pi = -64$. 
In our numerical
computations this lower bound appears to be sharp. In particular, we observe
that $U^n=1$ is a steady state whenever $\wD \geq -64$, but a boundary layer
forms already for e.g.\ $\wD = -64 - 10^{-8}$. 
As an example, we present a run
for $\wD = -65$ in Figure~\ref{fig:bl}, where we can clearly see how the 
boundary layer develops. Once the boundary layer has formed, the inner phase
first shrinks and then disappears, leading to the steady state solution
$U^0 = -1$ and $W^n = \wD$.
Note that this phenomenon is completely
independent from the choice of anisotropy $\gamma$. The discretization
parameters for this experiment were $N_f = N_c = 128$ and $\tau = 10^{-5}$
with $T=10^{-3}$.
\begin{figure}
\center
\includegraphics[angle=-0,width=0.19\textwidth]{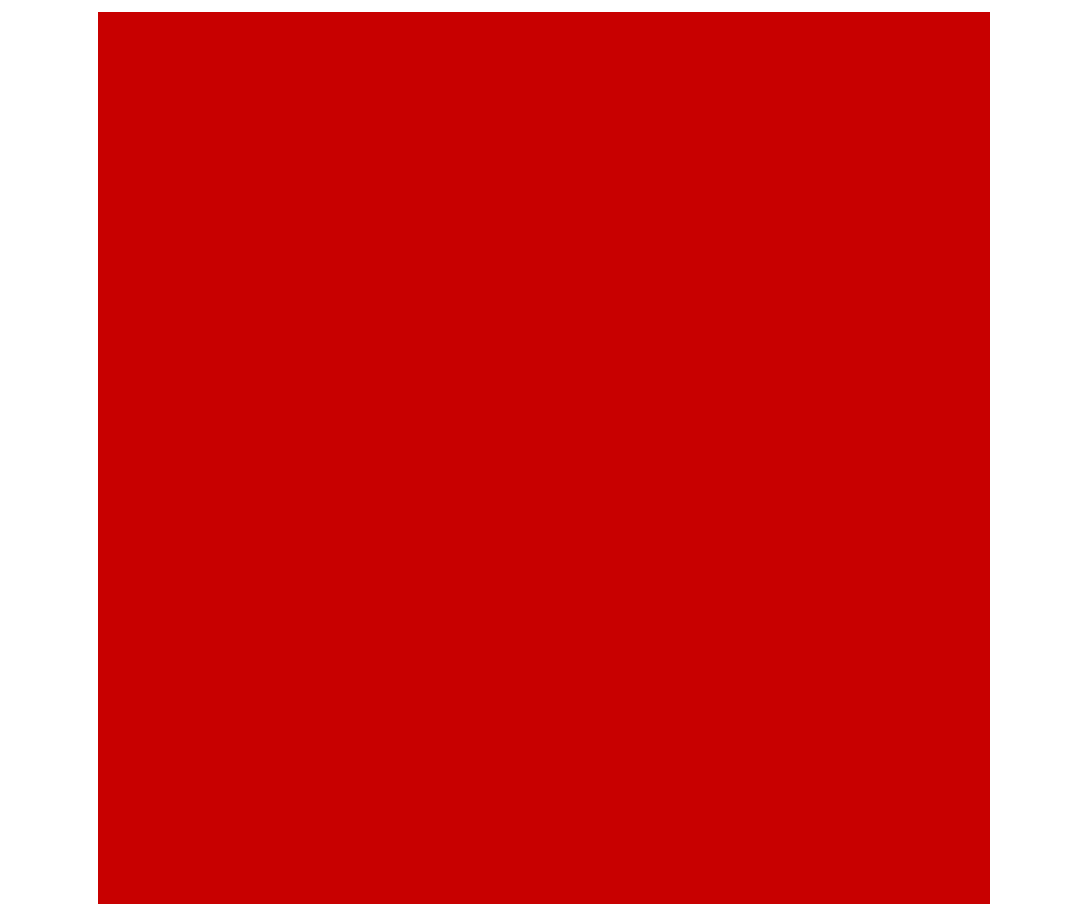}
\includegraphics[angle=-0,width=0.19\textwidth]{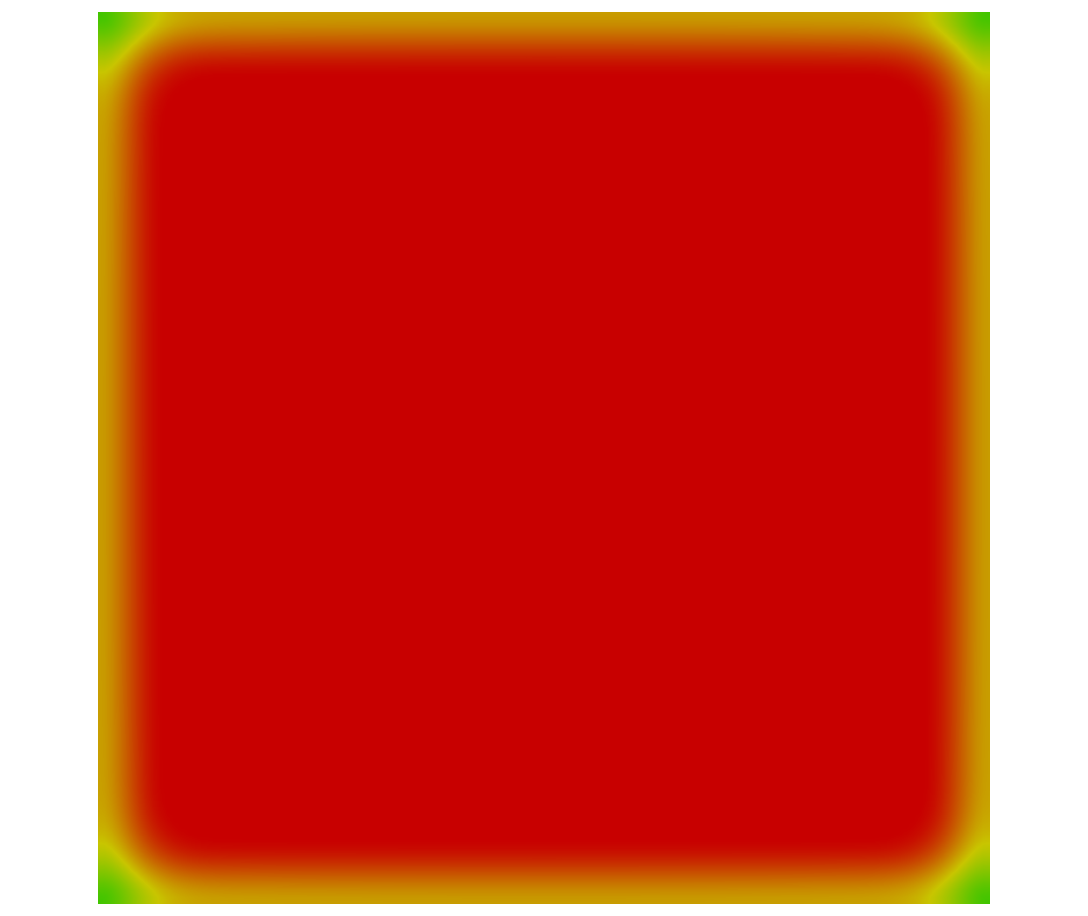}
\includegraphics[angle=-0,width=0.19\textwidth]{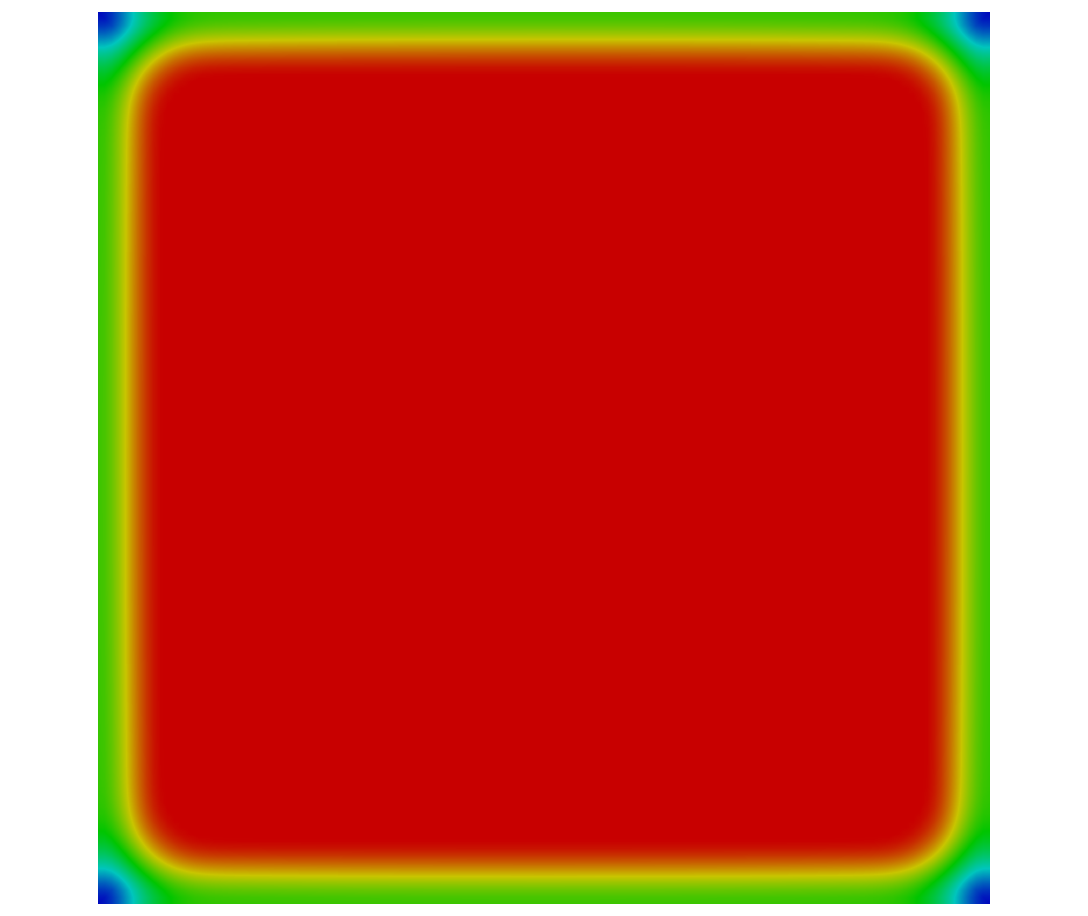}
\includegraphics[angle=-0,width=0.19\textwidth]{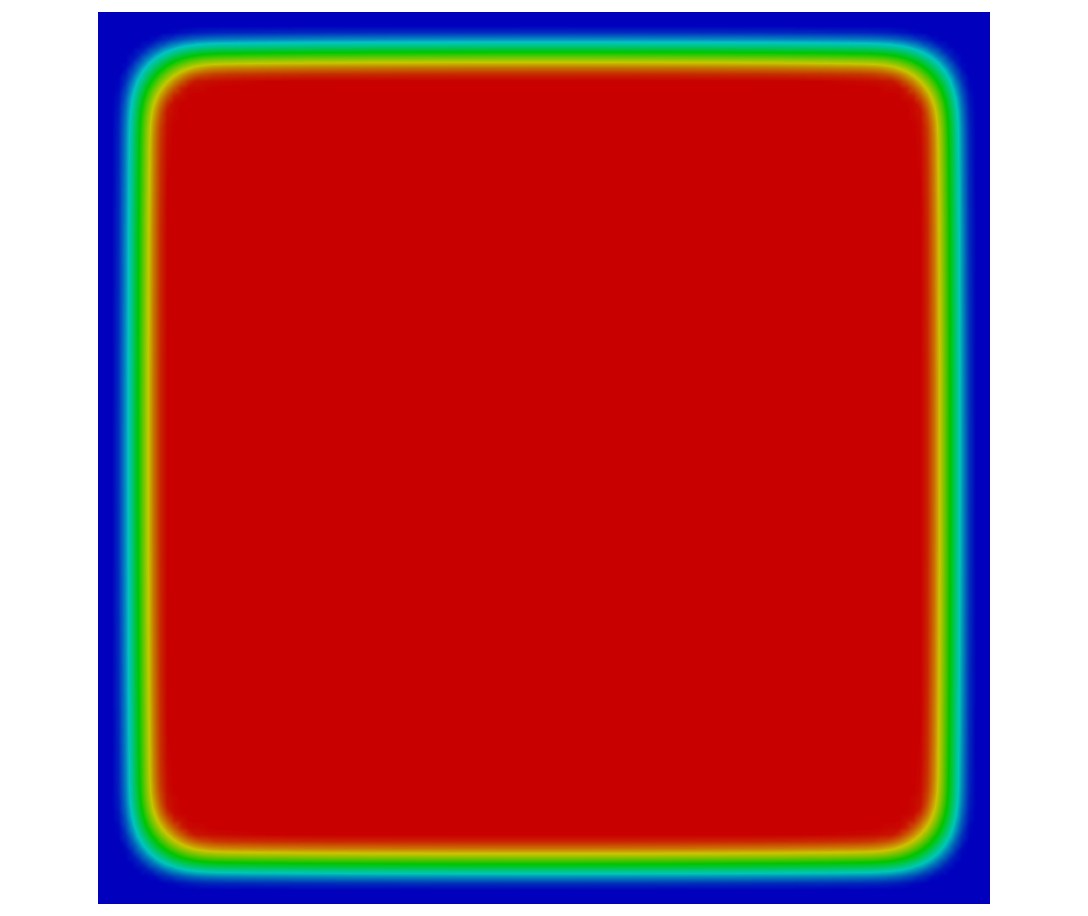}
\includegraphics[angle=-0,width=0.19\textwidth]{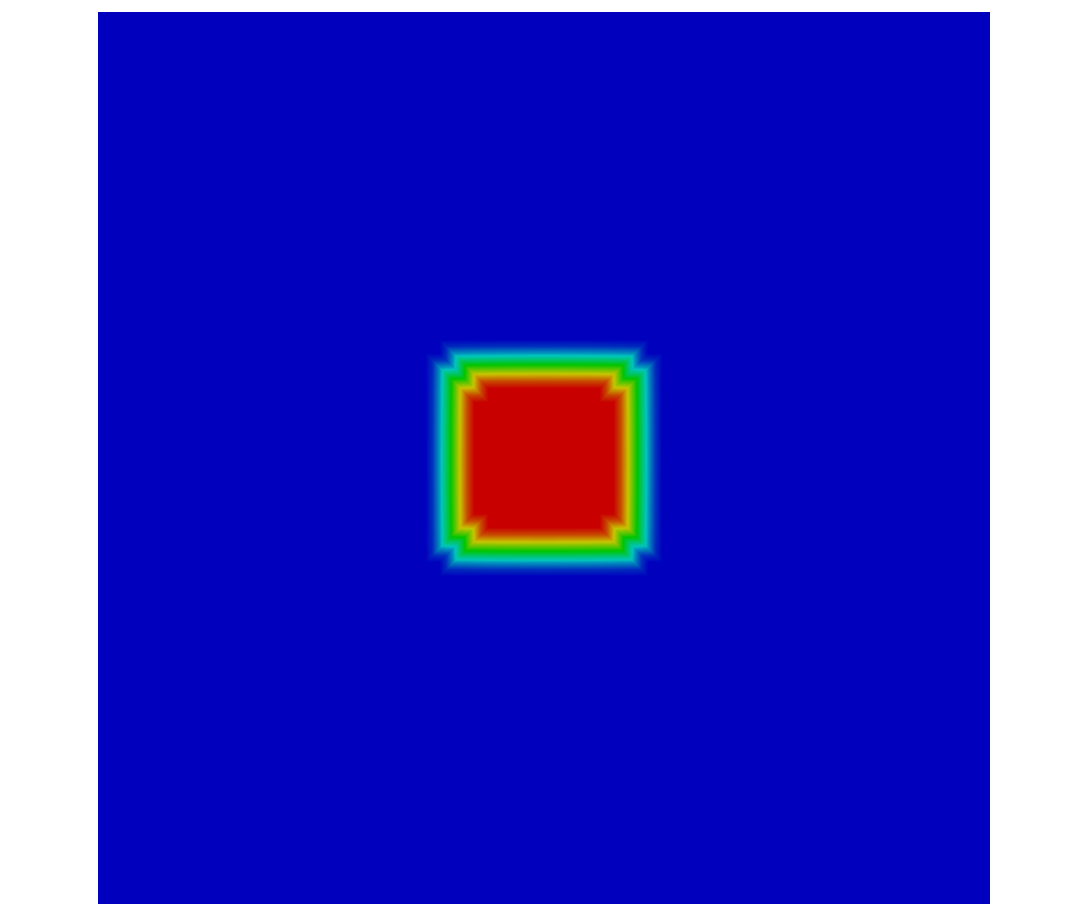}
\includegraphics[angle=-90,width=0.35\textwidth]{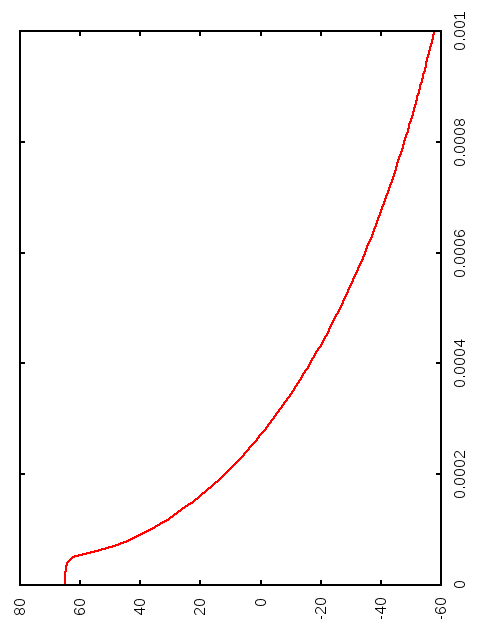}
\caption{({\sc ani$_1^{(0.01)}$}, $\wD = -65$)
Creation of a boundary layer.
Snapshots of the solution at times $t=0,\,4\times10^{-5},\,5\times10^{-5},\,
\,7\times10^{-5},\,10^{-3}$.
A plot of $\mathcal{F}_\gamma^h$ below.
}
\label{fig:bl}
\end{figure}%

Next we simulate the growth of a small crystal, when the sharp interface
evolution is given by (\ref{eq:MSa}--c), (\ref{eq:wD}).
In particular, we fix $H=8$, $\wD = -2$ and $\alpha = 0.03$; and we observe  
that for this choice of parameters the condition (\ref{eq:bl}) is satisfied
if we choose $\epsilon^{-1} = 32\,\pi > \frac{50}3\,\pi$.
A run for (\ref{eq:Ubc},b), when the initial seed has radius $0.1$,
with the discretization parameters $N_f = 4096$, $N_c = 128$, $\tau = 10^{-4}$
and $T=7.5$
is shown in Figure~\ref{fig:BSnew32pi}. We observe that at first the crystal
seed grows, taking on the form of the Wulff shape of $\gamma$. Then the four
sides break and become nonconvex, with the four side arms that grow at the 
corners yielding a shape that is well-known in the numerical simulation of
dendritic growth. 
\begin{figure}
\center
\includegraphics[angle=-0,width=0.19\textwidth]{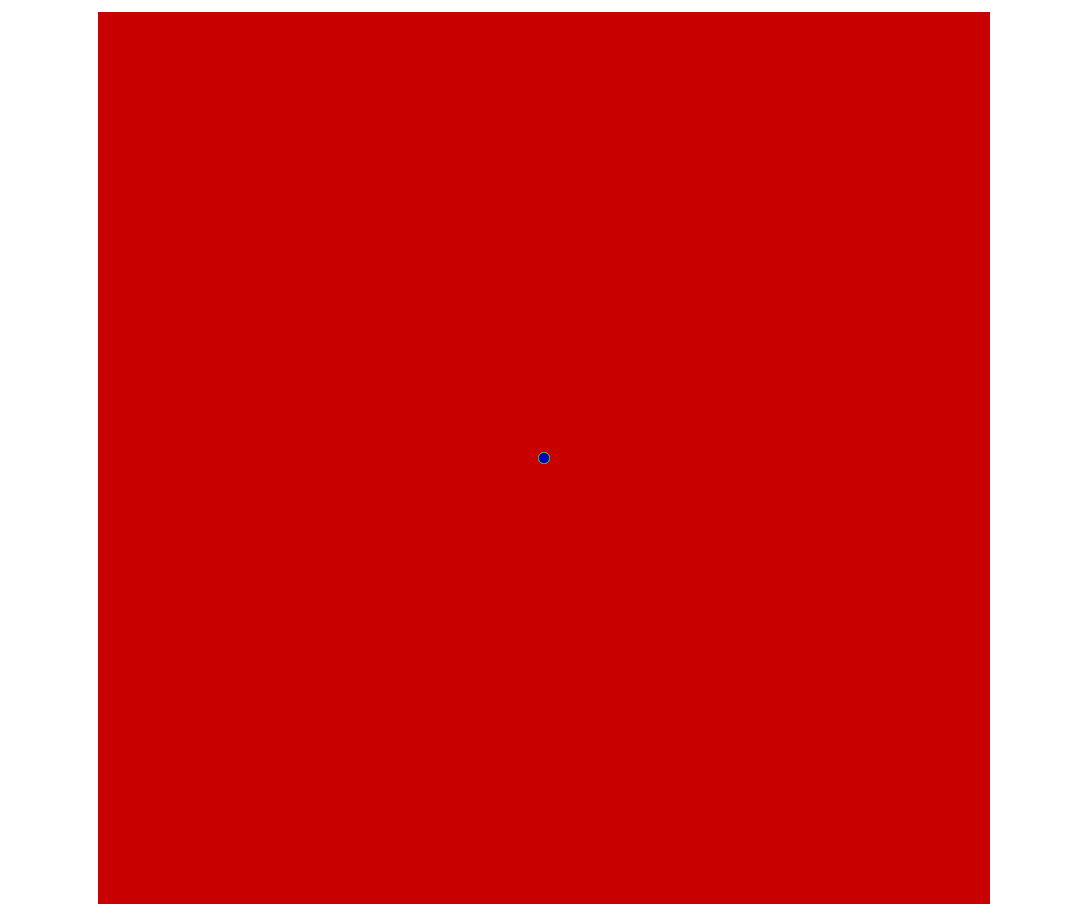}
\includegraphics[angle=-0,width=0.19\textwidth]{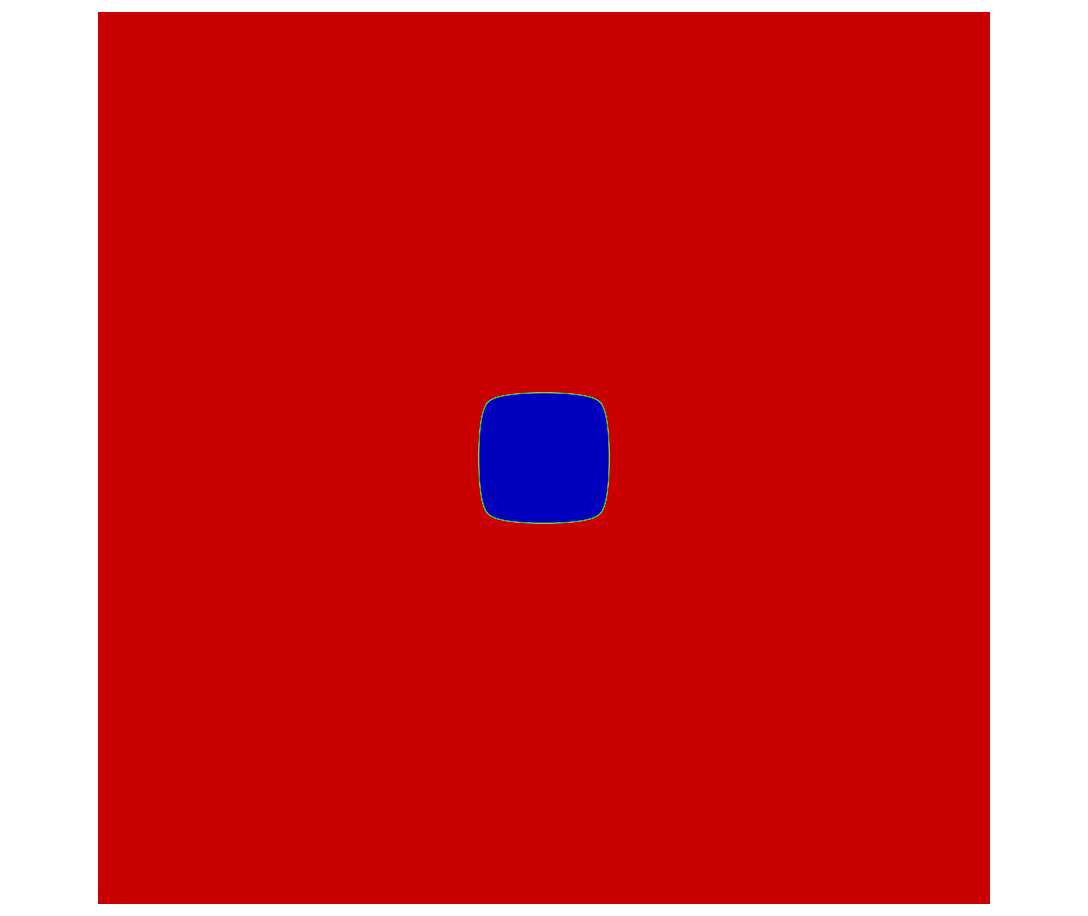}
\includegraphics[angle=-0,width=0.19\textwidth]{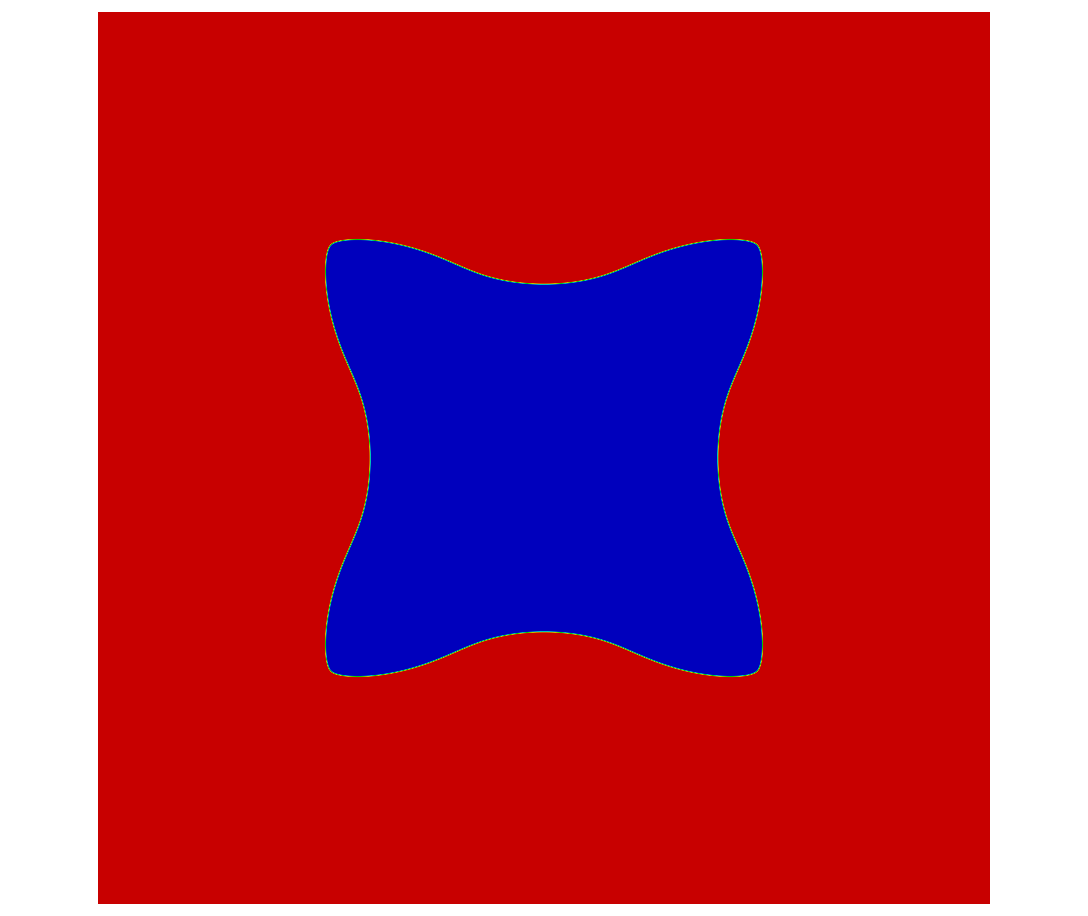}
\includegraphics[angle=-0,width=0.19\textwidth]{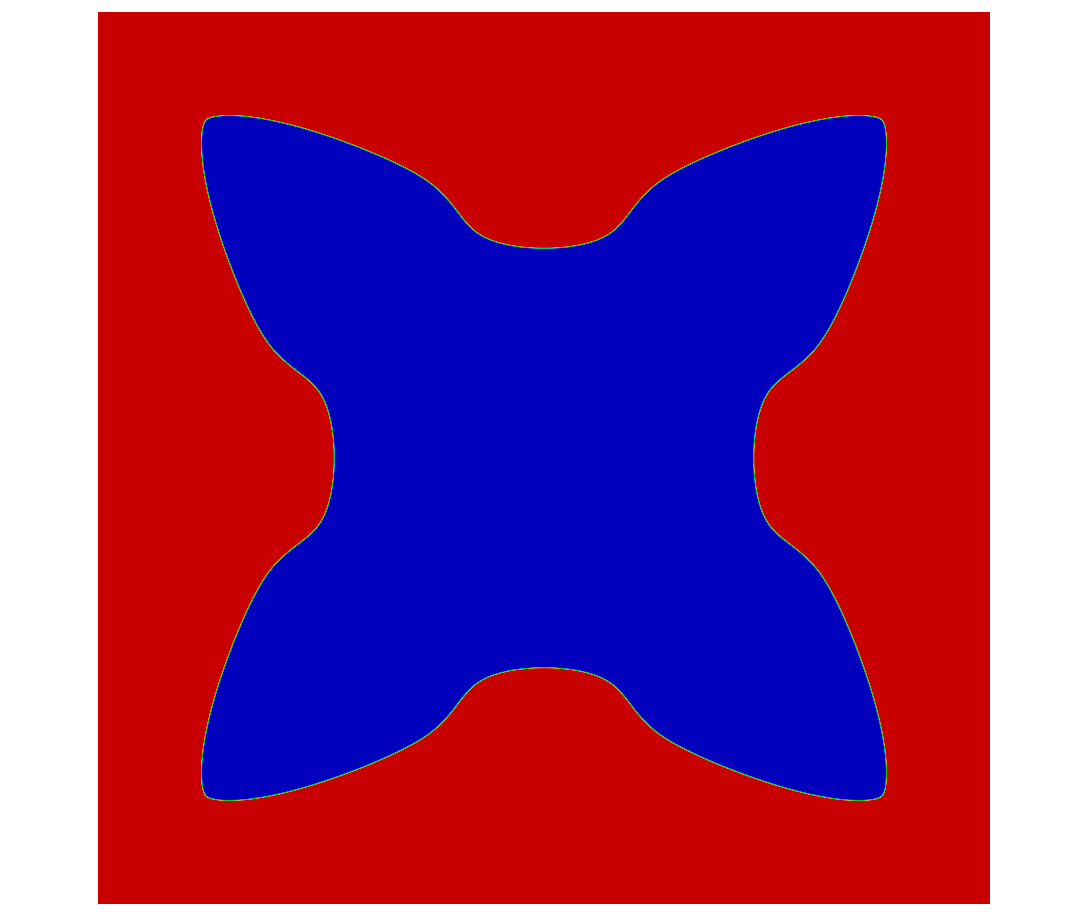}
\includegraphics[angle=-0,width=0.19\textwidth]{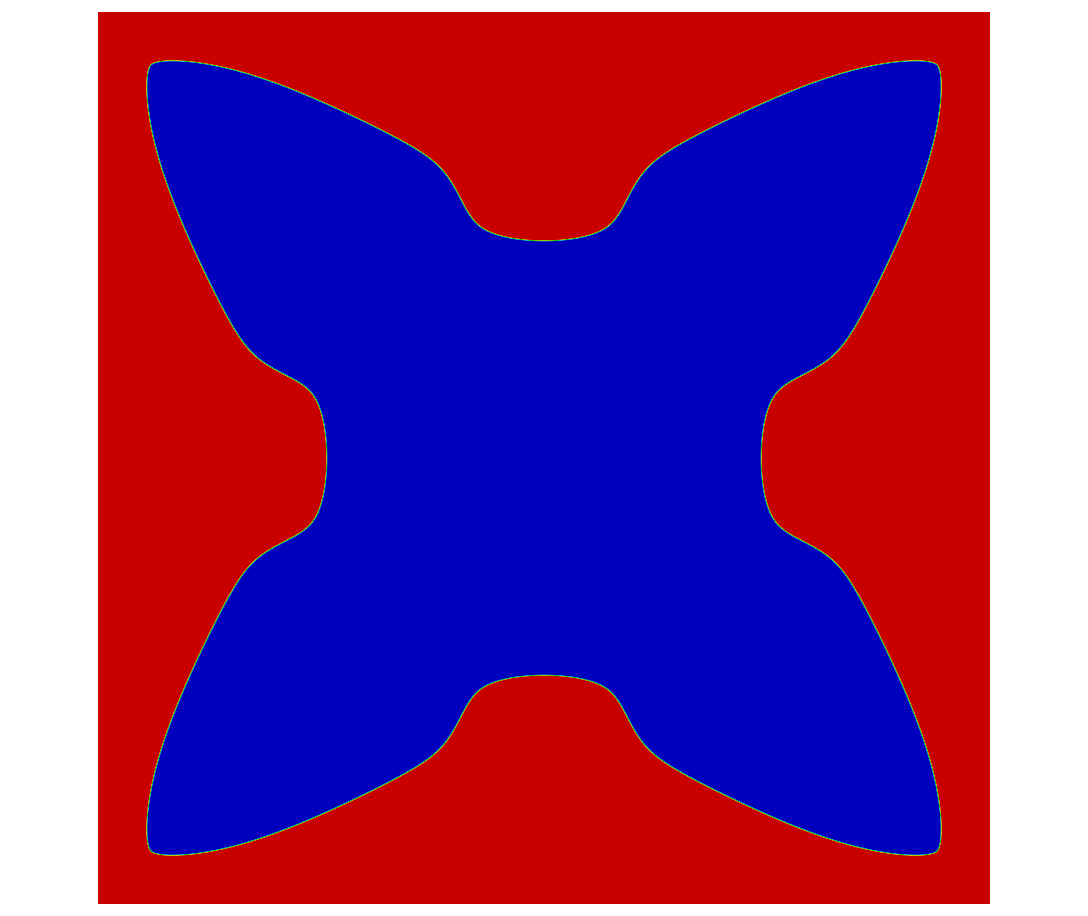}
\includegraphics[angle=-90,width=0.35\textwidth]{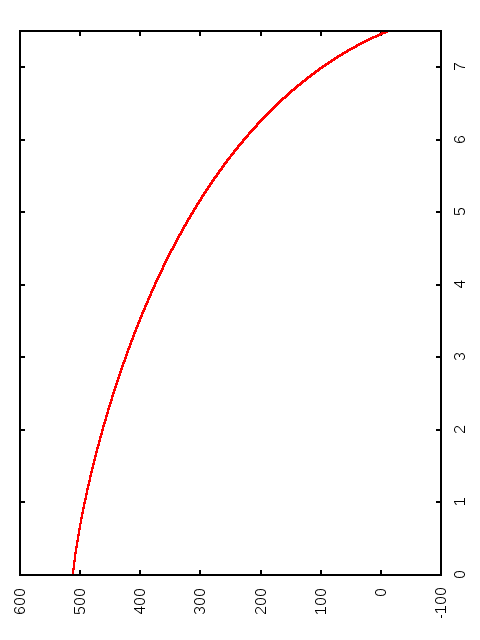}
\caption{({\sc ani$_1^{(0.3)}$}, 
$\wD = -2$, $\epsilon^{-1} = 32\,\pi$, $\Omega=(-8,8)^2$)
A phase field approximation for the anisotropic Mullins--Sekerka problem
(\ref{eq:MSa}--c), (\ref{eq:wD}). 
Snapshots of the solution at times $t=0,\,1,\,5,\,7,\,7.5$.
A plot of $\mathcal{F}_{\gamma}^h$ below.
}
\label{fig:BSnew32pi}
\end{figure}%

\subsection{Numerical results in 3d}
A numerical experiment for (\ref{eq:aMC}) in 3d with the help of  
the approximation (\ref{eq:U2}), (\ref{eq:W}) for
the Allen--Cahn equation (\ref{eq:ACa}--c) can be seen in 
Figure~\ref{fig:aMC3d}.
Here the initial profile is given by a sphere with radius $0.3$. We set
$\tau=10^{-4}$ and $T=0.03$. It can be seen that the initially round sphere
assumes the cylindrical Wulff shape as it shrinks, before the interface shrinks
to a point and disappears completely.
\begin{figure}
\center
\includegraphics[angle=-0,width=0.19\textwidth]{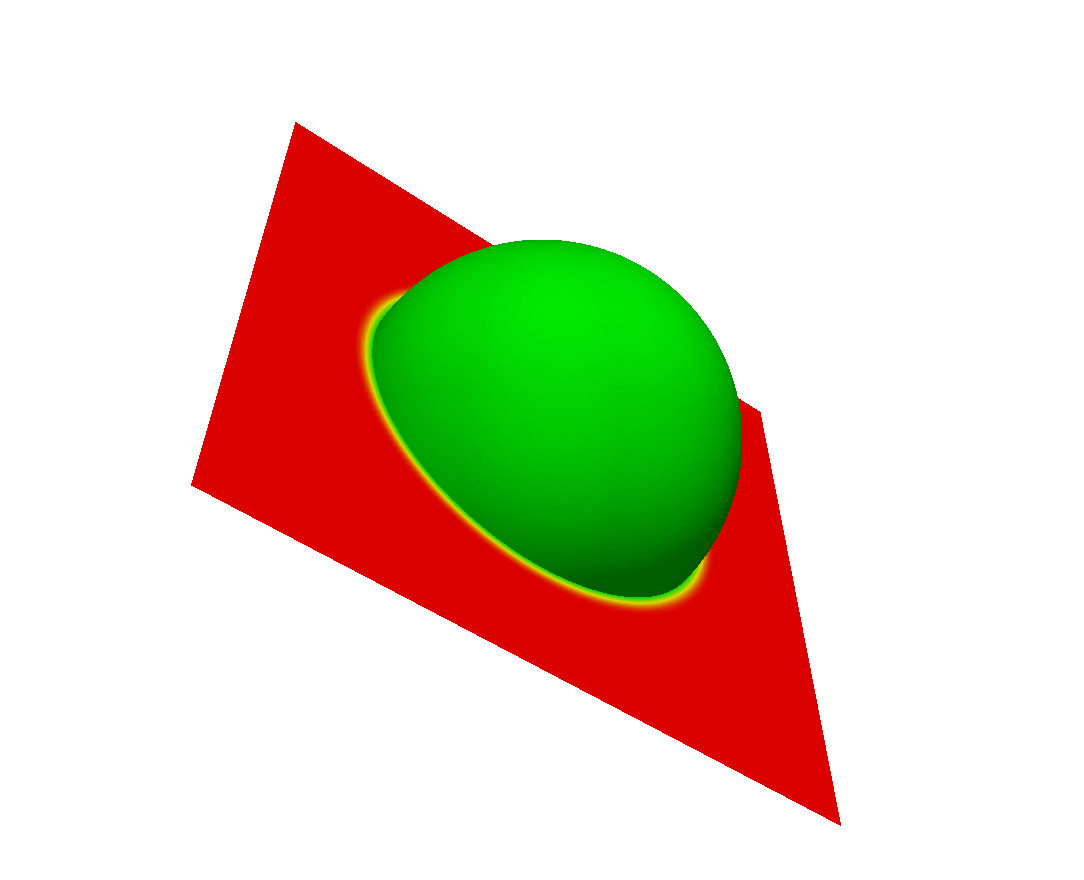}
\includegraphics[angle=-0,width=0.19\textwidth]{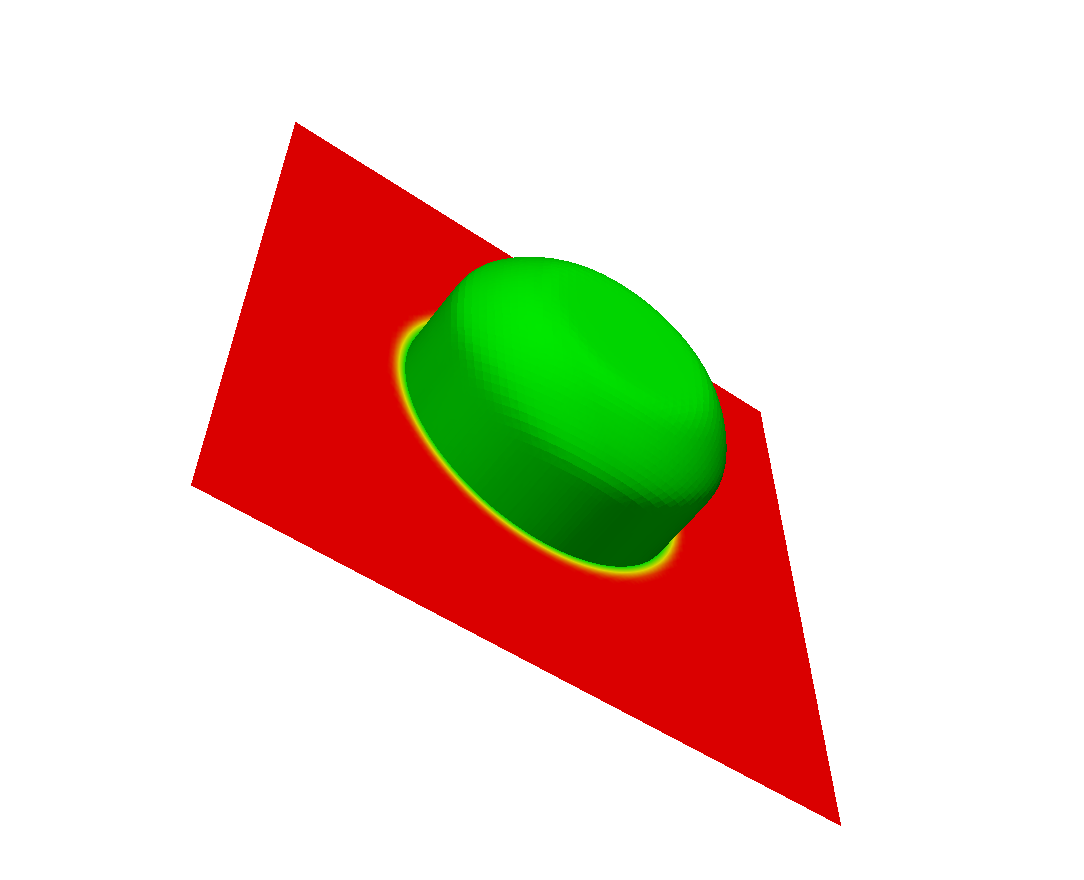}
\includegraphics[angle=-0,width=0.19\textwidth]{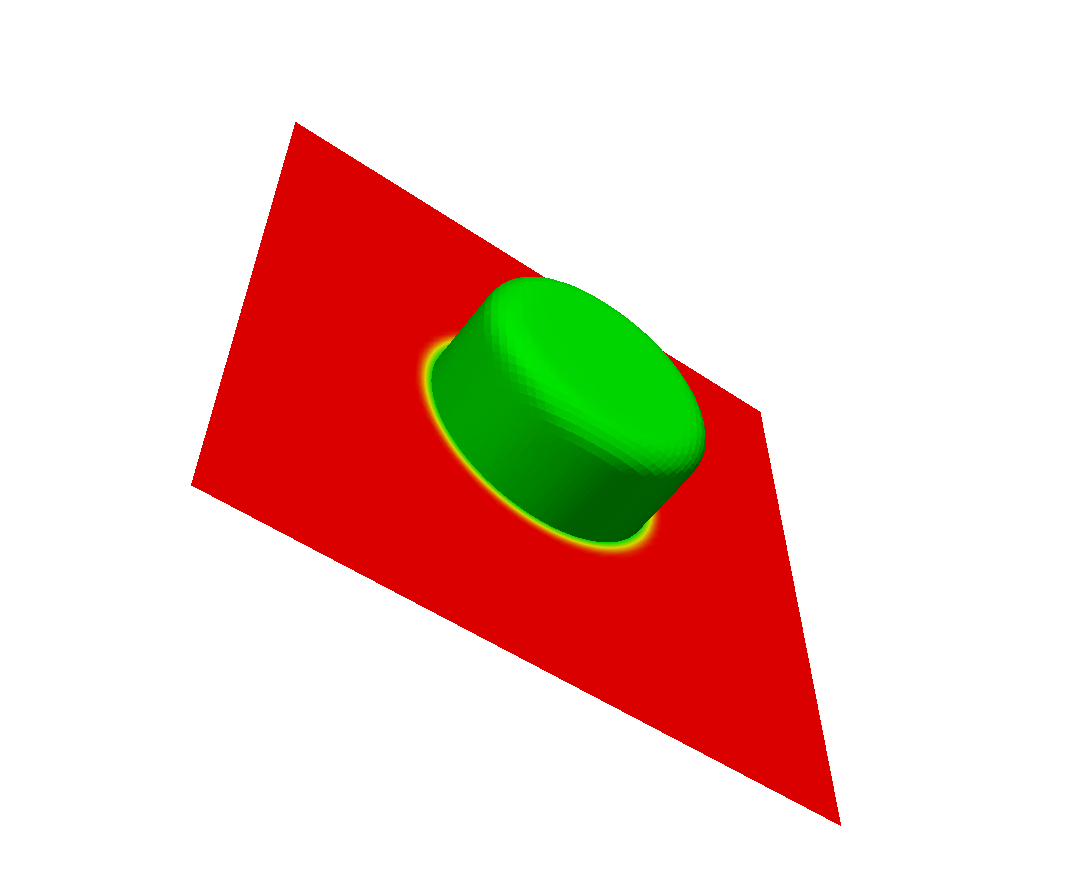}
\includegraphics[angle=-0,width=0.19\textwidth]{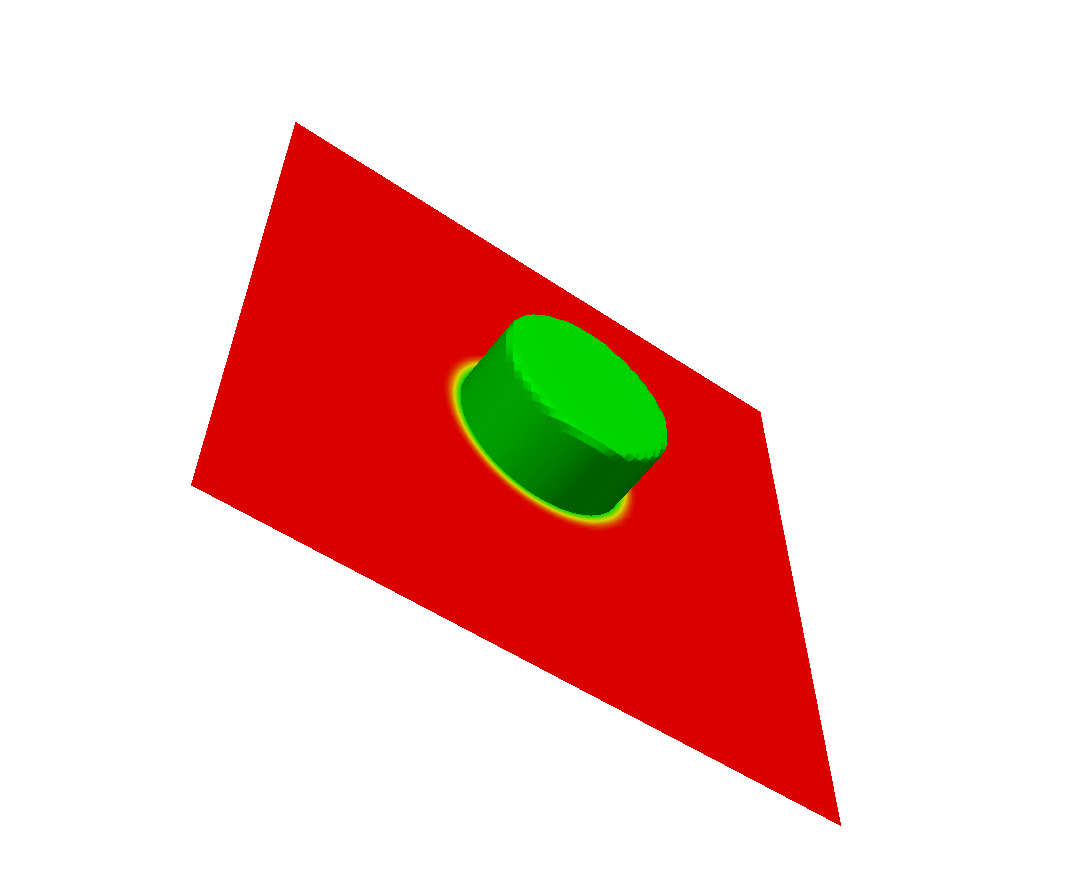}
\includegraphics[angle=-0,width=0.19\textwidth]{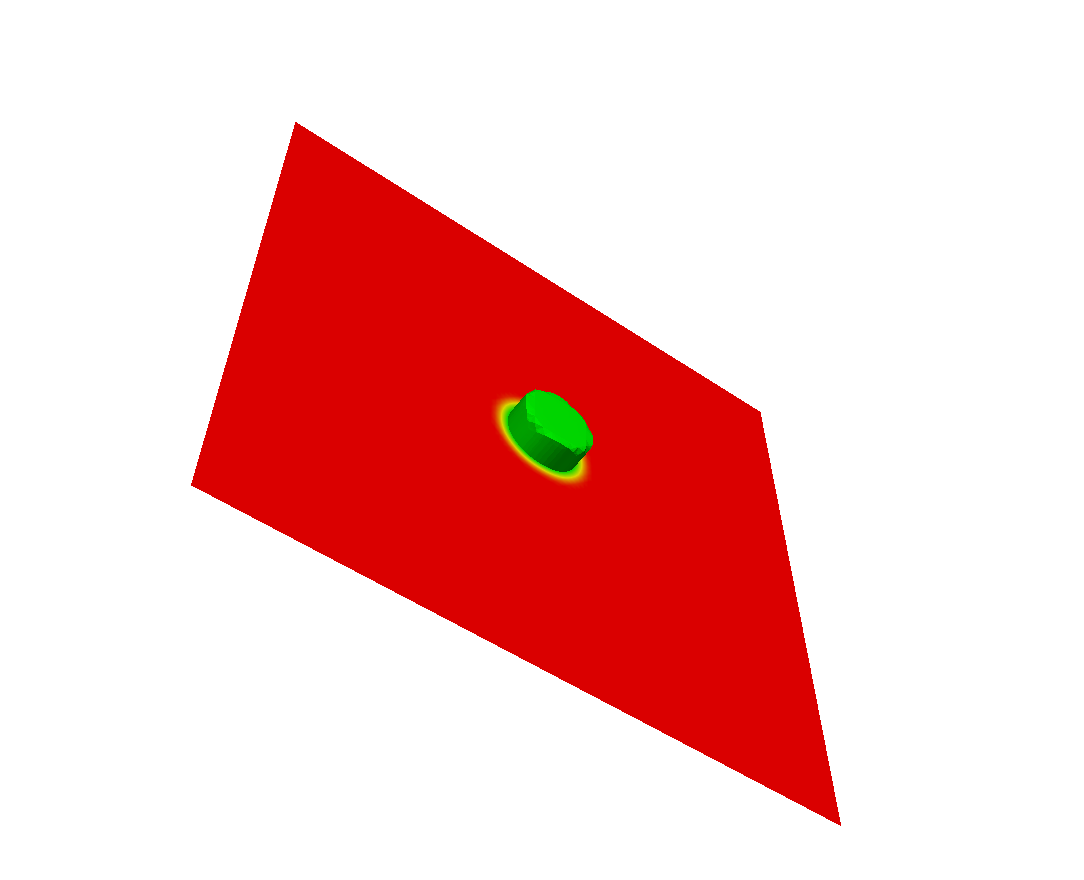}
\includegraphics[angle=-90,width=0.35\textwidth]{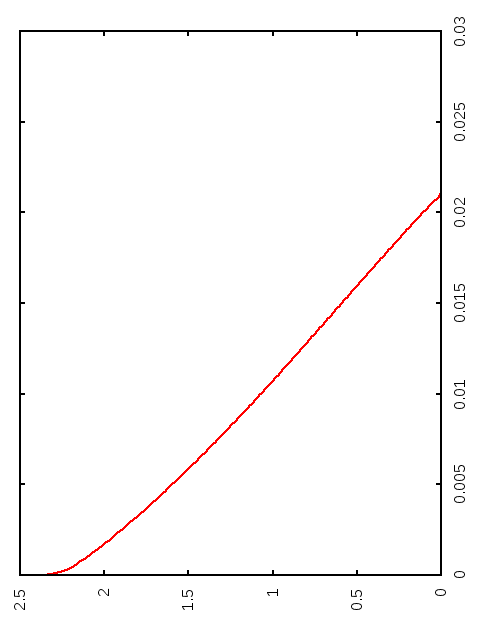}
\caption{({\sc ani$_2$}) 
A phase field approximation for the anisotropic mean curvature flow
(\ref{eq:aMC}).
Snapshots of the solution at times $t=0,\,5\times10^{-3},\,10^{-2},\,
1.5\times10^{-2},\,2\times10^{-2}$.
A plot of $\mathcal{E}_\gamma^h$ below.
}
\label{fig:aMC3d}
\end{figure}%

A numerical experiment for (\ref{eq:aSD}) with the help of
the approximation (\ref{eq:U},b) for
the Cahn--Hilliard equation (\ref{eq:CHa}--d) 
can be seen in Figure~\ref{fig:hexSD3d}.
Here the initial profile is given by a sphere with radius $0.3$. 
We set $\tau=10^{-6}$ and $T=10^{-3}$. We can clearly see the 
evolution from the round sphere to the strongly facetted Wulff shape.
\begin{figure}
\center
\includegraphics[angle=-0,width=0.19\textwidth]{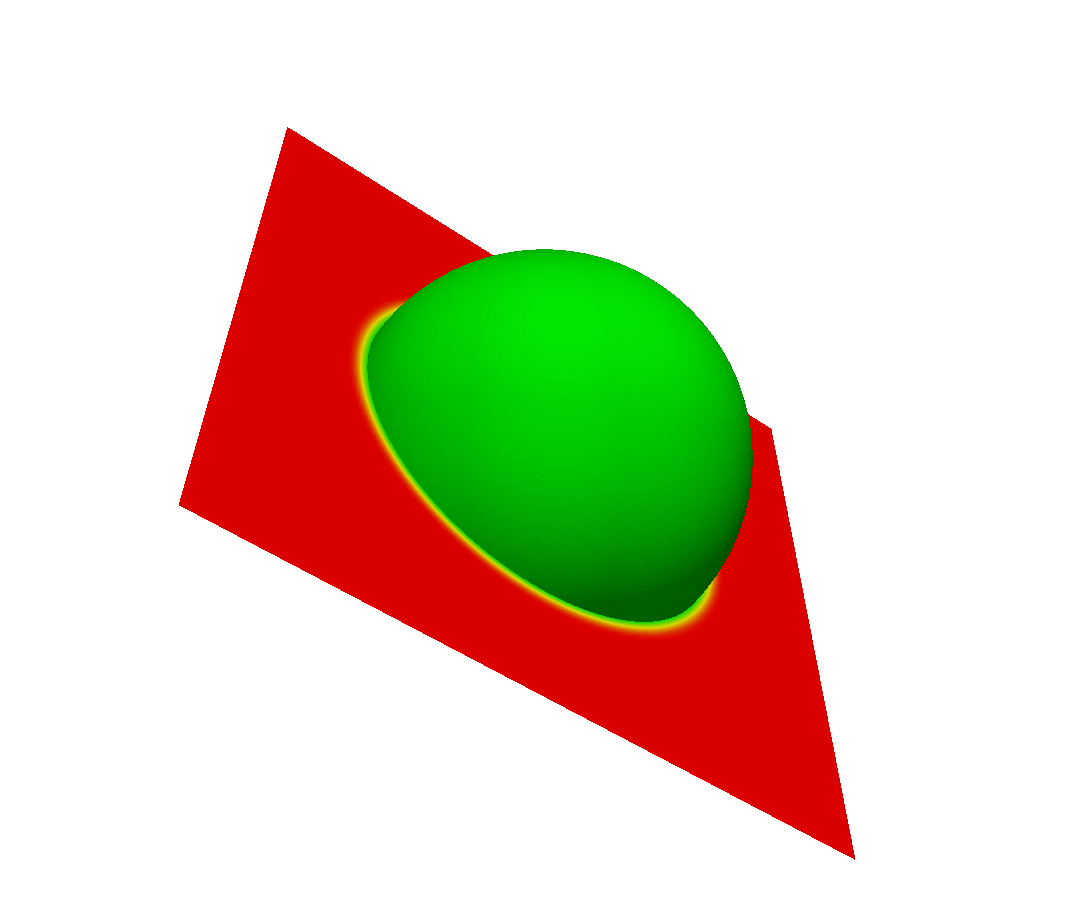}
\includegraphics[angle=-0,width=0.19\textwidth]{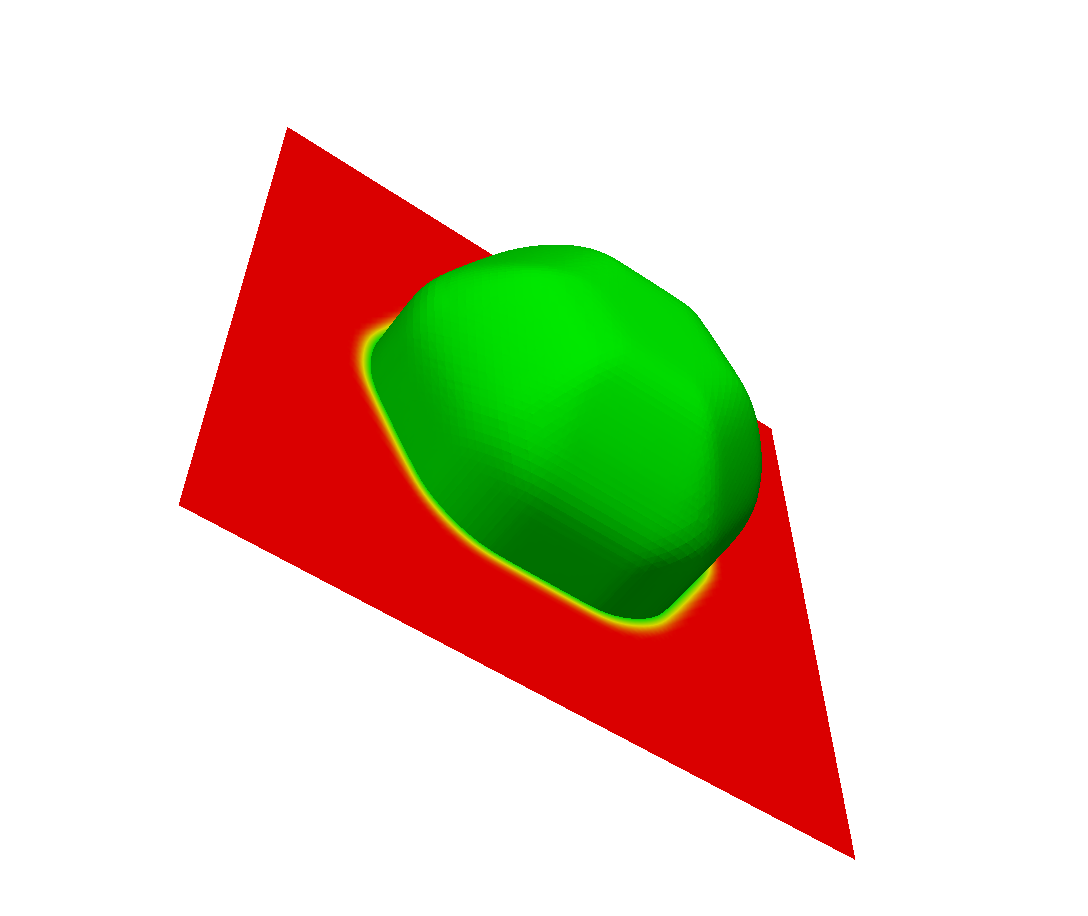}
\includegraphics[angle=-0,width=0.19\textwidth]{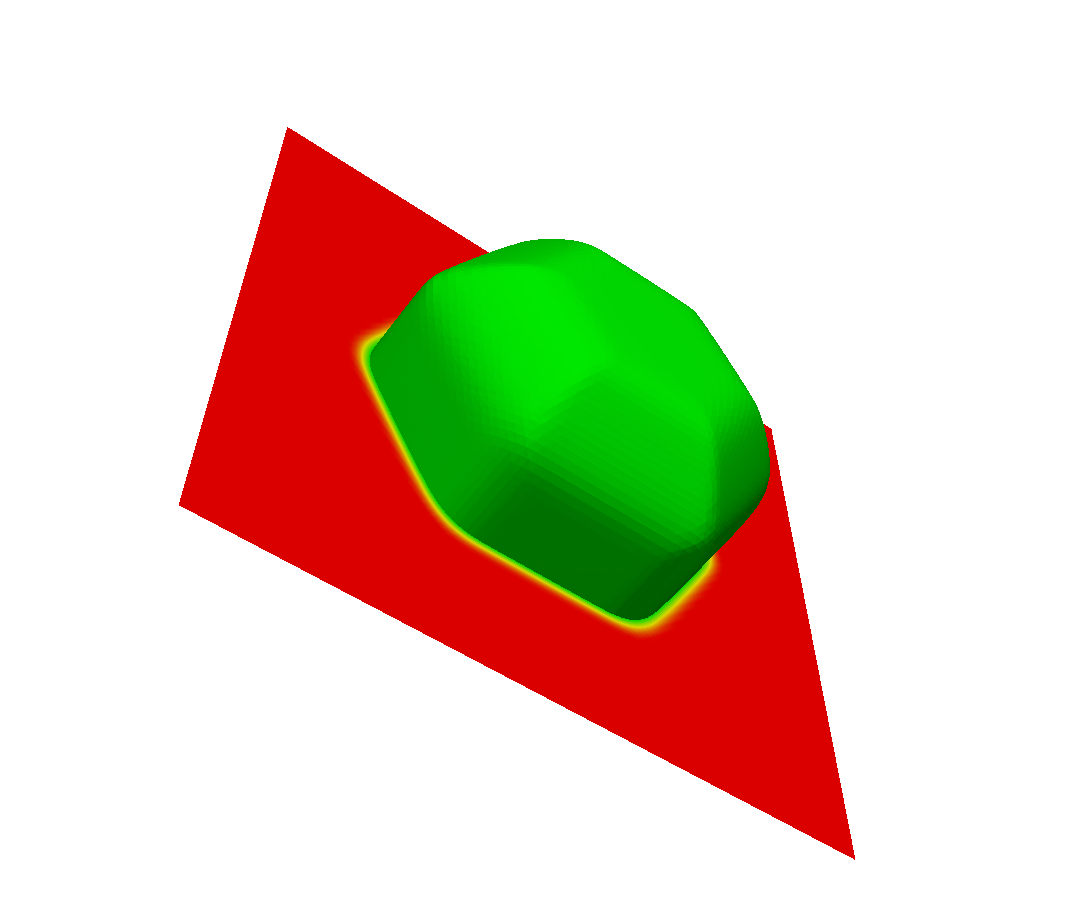}
\includegraphics[angle=-0,width=0.19\textwidth]{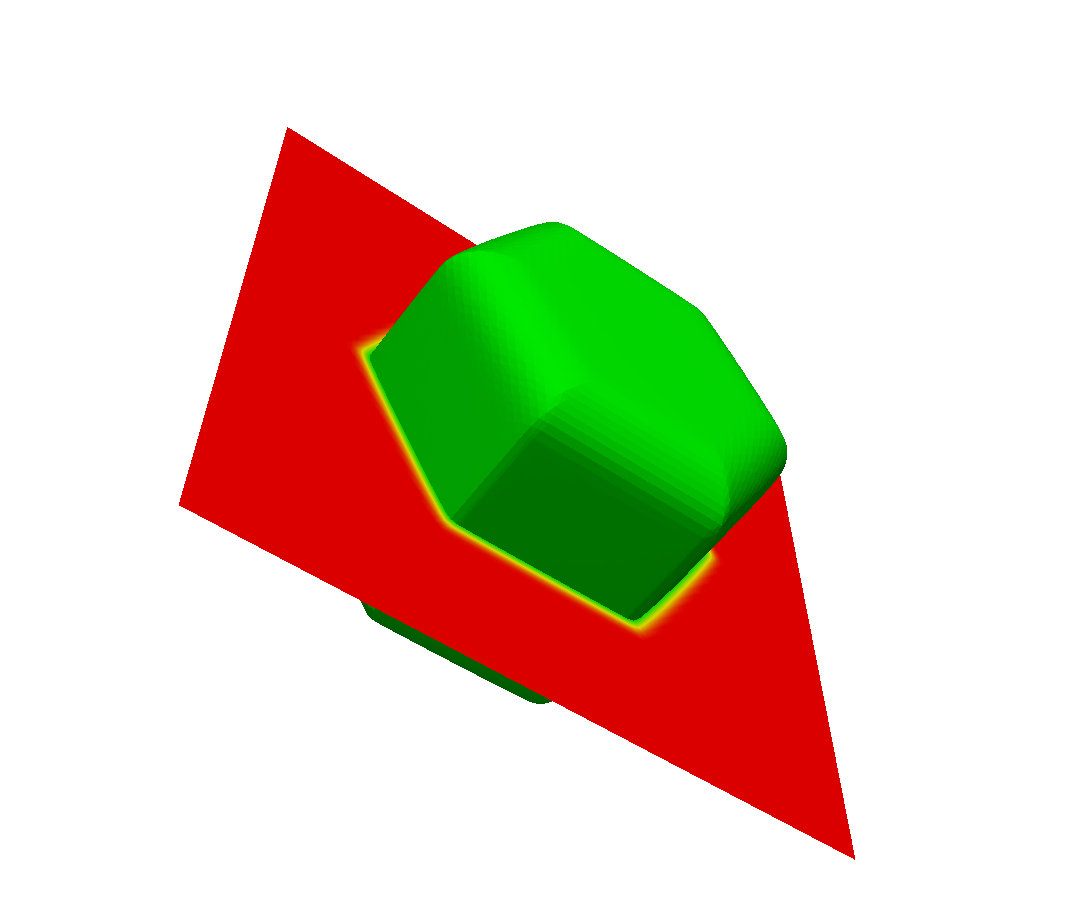}
\includegraphics[angle=-0,width=0.19\textwidth]{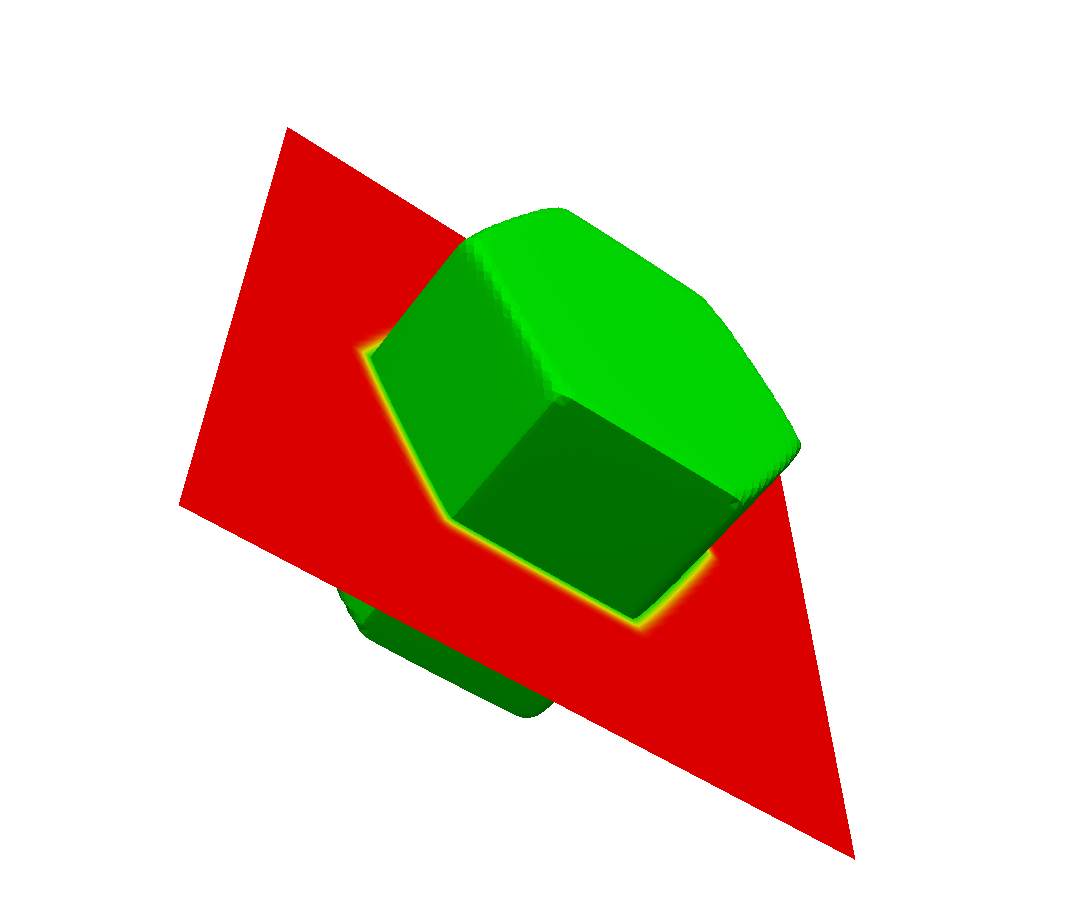}
\includegraphics[angle=-90,width=0.35\textwidth]{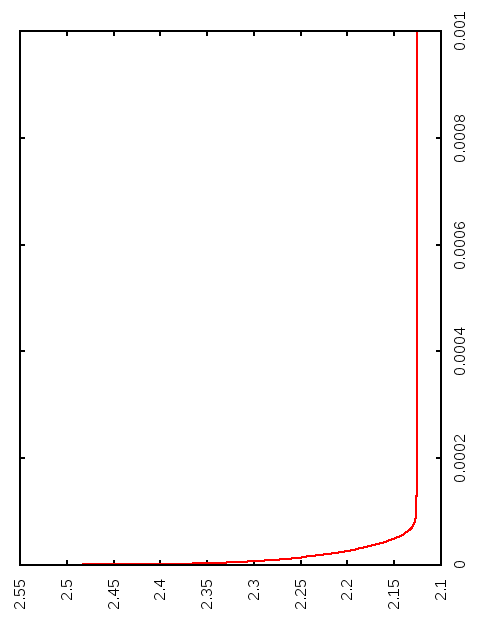}
\caption{({\sc ani$_4$})
A phase field approximation for anisotropic surface diffusion (\ref{eq:aSD}).
Snapshots of the solution at times $t=0,\,10^{-5},\,
2\times10^{-5},\,5\times10^{-5},\, 10^{-3}$.
A plot of $\mathcal{E}_\gamma^h$ below.
}
\label{fig:hexSD3d}
\end{figure}%

A numerical approximation for the sharp interface problem (\ref{eq:MSa}--d), 
is shown in Figure~\ref{fig:3daMS}. Here the initial interface is given by the
boundary of a $8\times1\times1$ cuboid with minor side length $0.1$, and
we set $\tau=10^{-5}$ and $T=5\times10^{-3}$. 
We can observe that during the evolution the elongated 
facets become bent and nonconvex, before the solution converges to the Wulff 
shape.
\begin{figure}
\center
\includegraphics[angle=-0,width=0.19\textwidth]{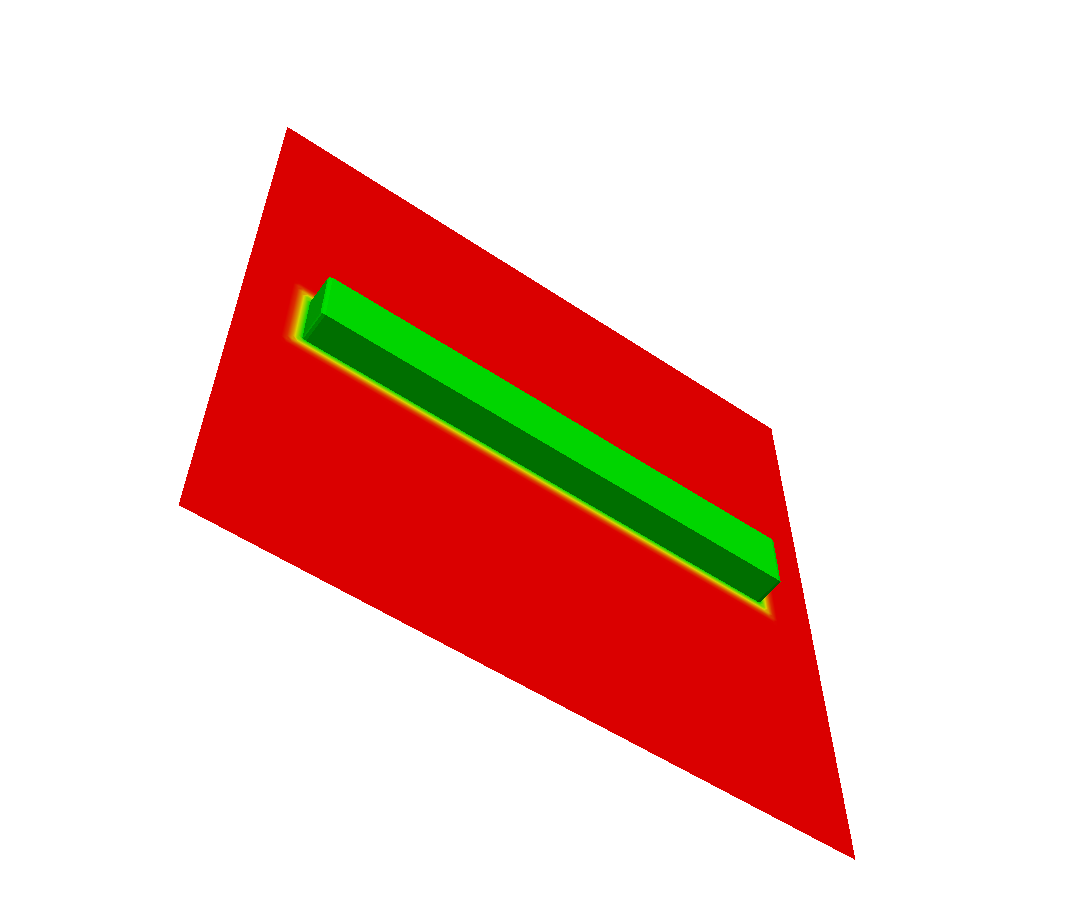}
\includegraphics[angle=-0,width=0.19\textwidth]{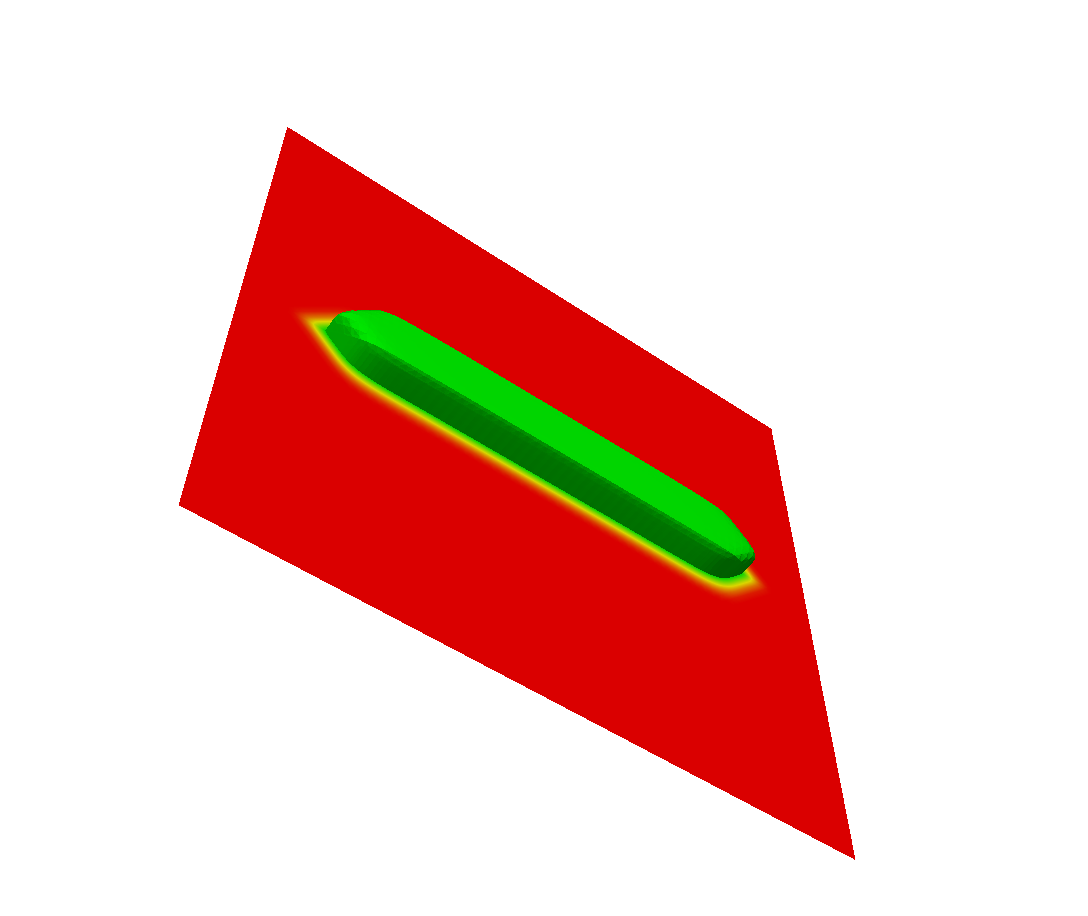}
\includegraphics[angle=-0,width=0.19\textwidth]{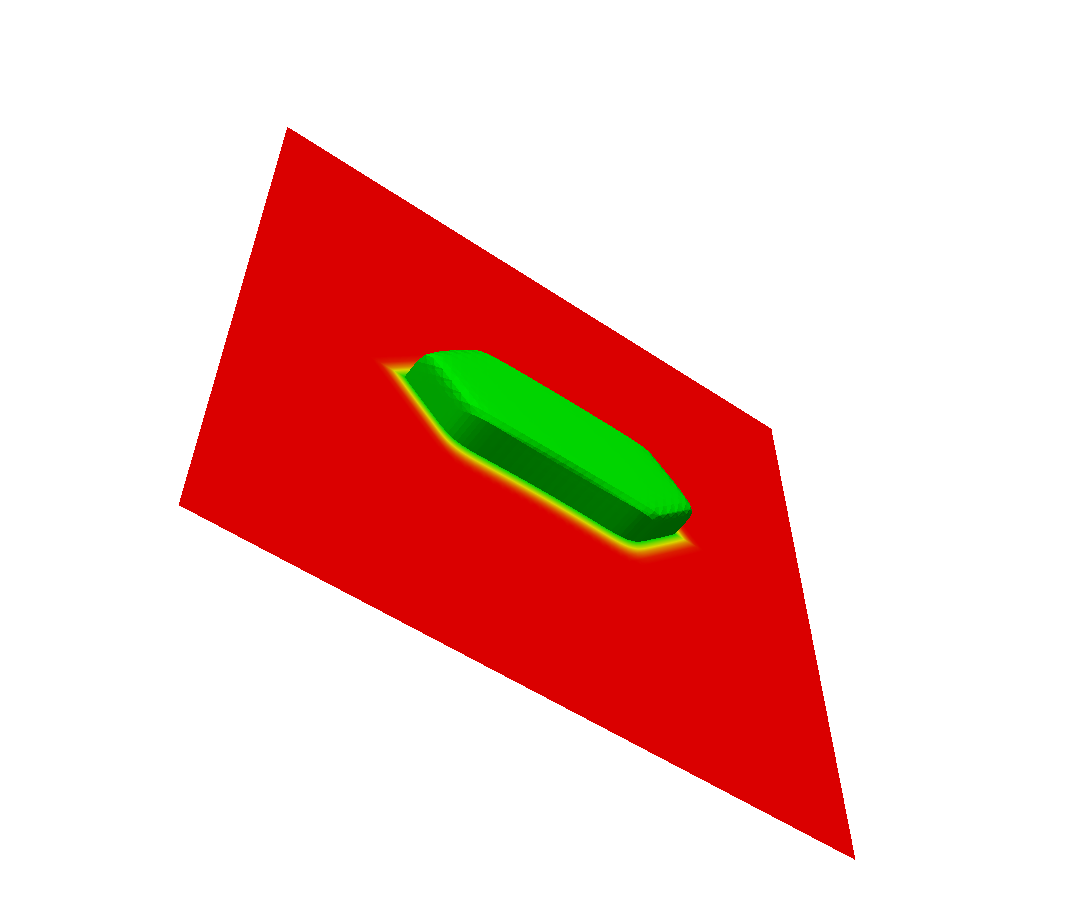}
\includegraphics[angle=-0,width=0.19\textwidth]{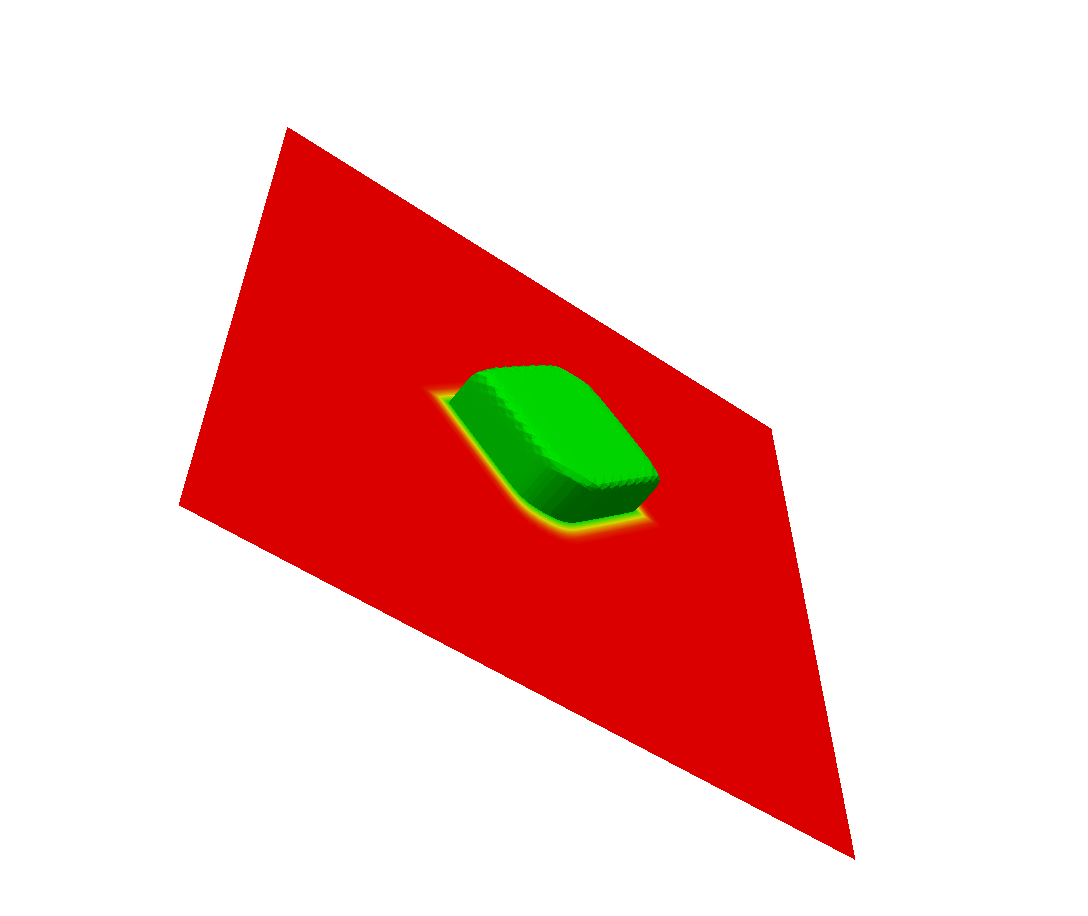} 
\includegraphics[angle=-0,width=0.19\textwidth]{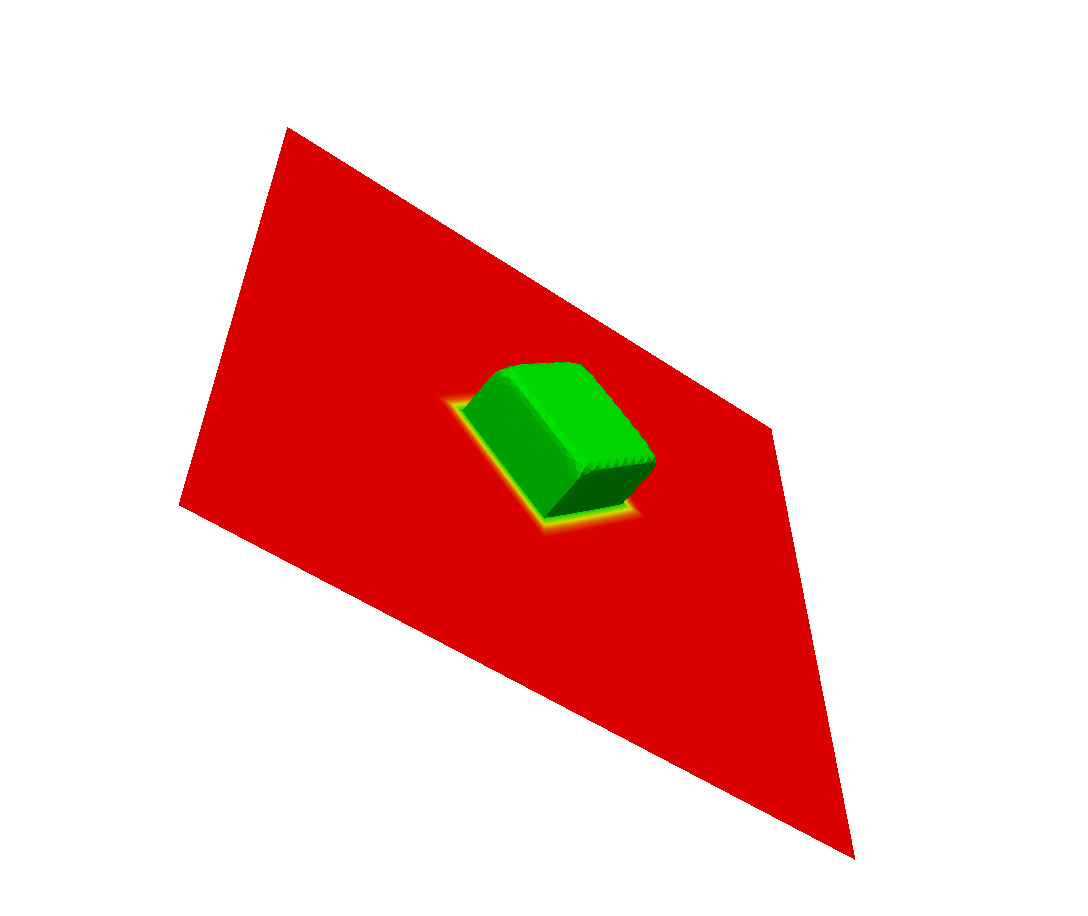}
\includegraphics[angle=-0,width=0.19\textwidth]{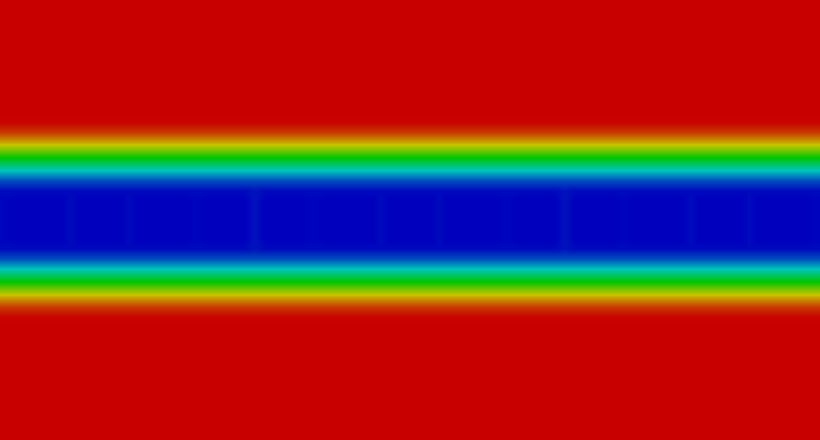}
\includegraphics[angle=-0,width=0.19\textwidth]{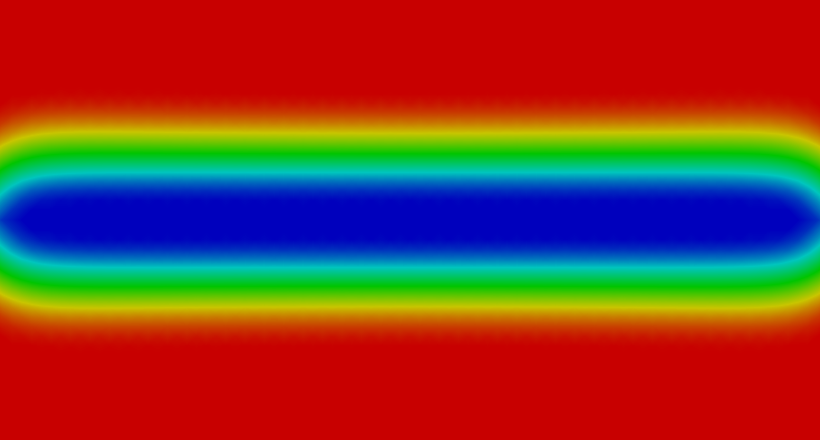}
\includegraphics[angle=-0,width=0.19\textwidth]{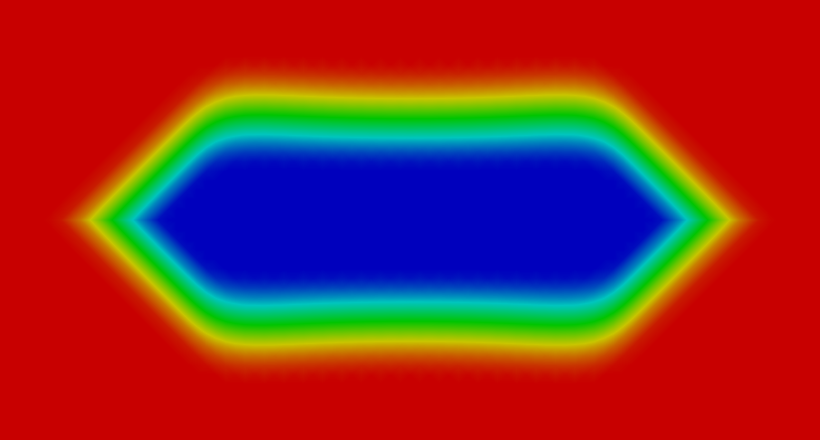}
\includegraphics[angle=-0,width=0.19\textwidth]{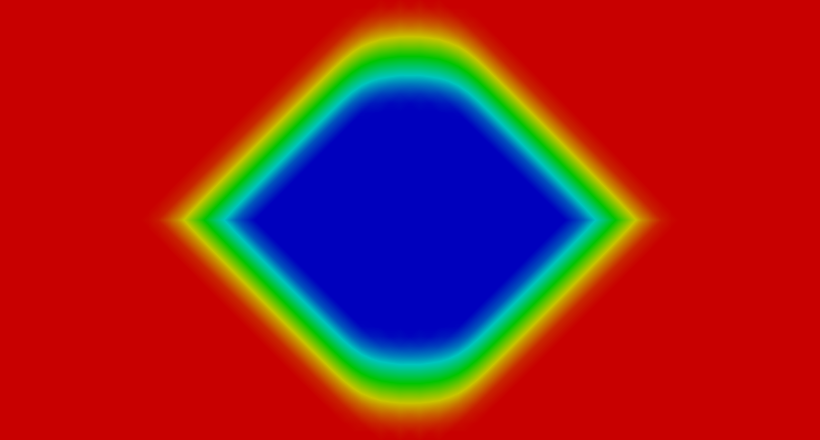}
\includegraphics[angle=-0,width=0.19\textwidth]{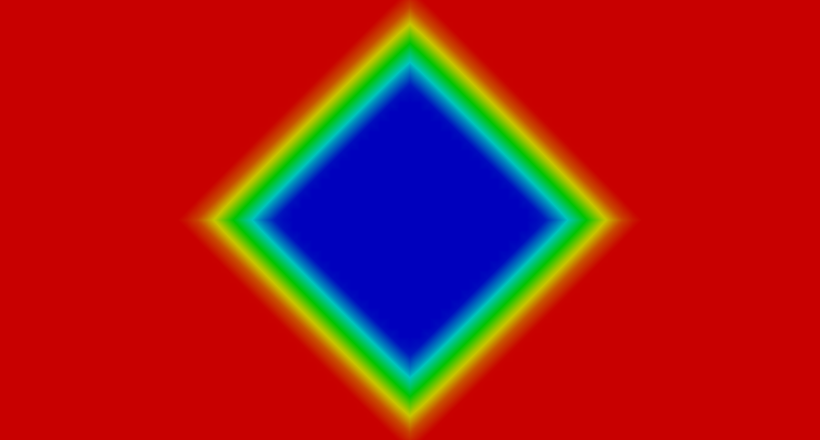}
\includegraphics[angle=-90,width=0.35\textwidth]{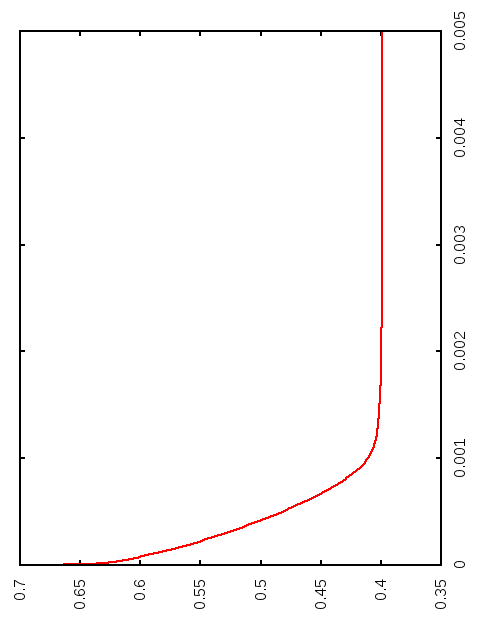}
\caption{({\sc ani$_1^\star$})
A phase field approximation for the anisotropic Mullins--Sekerka problem
(\ref{eq:MSa}--d).
Snapshots of the solution at times $t=0,\,10^{-4},\,5\times10^{-4},\,
10^{-3},\,5\times10^{-3}$. The middle row shows detailed $2d$ plots of the 
solution in the $x_1$--$x_2$ plane.
A plot of $\mathcal{E}_\gamma^h$ below.
}
\label{fig:3daMS}
\end{figure}%


\end{document}